\documentclass[11pt]{article}
\usepackage[margin=1in]{geometry}
\usepackage[pdftex]{graphicx}
\usepackage{multirow}
\usepackage{setspace}
\usepackage{combelow}
\usepackage{hyperref}
\usepackage{sectsty}
\usepackage{amsmath}
\usepackage{amsthm}
\usepackage{amssymb}
\usepackage{xcolor}
\usepackage{tikz}      % core TikZ
\usepackage{tikz-cd}
\usepackage{graphicx}
\usepackage{mathdots}
\usepackage{subcaption}
\usepackage{enumitem}
\usepackage{float}
\usepackage{amscd}
\usepackage{comment}
\usepackage{mathtools}

\usetikzlibrary{decorations.pathreplacing}
\usetikzlibrary{cd}
\newtheorem{example}{Example}
\newtheorem{theorem}{Theorem}

\newtheorem{definition}{Definition}
\newtheorem{proposition}[theorem]{Proposition}
\newtheorem{lemma}[theorem]{Lemma}

\newtheorem{corollary}[theorem]{Corollary}
\newtheorem{remark}[theorem]{Remark}

\newcommand{\BR}{\mathbb{R}}

\newcommand{\BC}{\mathbb{C}}

\newcommand{\BQ}{\mathbb{Q}}

\newcommand{\pro}{\operatorname{Coh^b_{pro}}}
\newcommand{\segal}{\mathcal{S}}

\newcommand{\fbeta}{\mathcal{F}^{\beta}}
\newcommand{\tbeta}{\mathcal{T}^{\beta}}
\newcommand{\abeta}{\mathcal{A}^{\beta}}
\newcommand{\abetator}{\mathcal{A}^{\beta}_{\operatorname{tor}}}
\newcommand{\abetatf}{\mathcal{A}^{\beta}_{\operatorname{t.f.}}}

\newcommand{\fbetastack}{\boldsymbol{\operatorname{Coh}}_{\operatorname{t.f.}}(X)}
\newcommand{\tbetastack}{\boldsymbol{\operatorname{Coh}}_{\operatorname{tor}}(X)}
\newcommand{\abetatorstack}{\boldsymbol{\operatorname{Coh}}_{\operatorname{tor}}(X,\tau^{\beta})}
\newcommand{\abetatfstack}{\boldsymbol{\operatorname{Coh}}_{\operatorname{t.f.}}(X,\tau^{\beta})}

\newcommand{\coh}{\operatorname{Coh}}
\newcommand{\cohbetass}{\operatorname{Coh}^{ss}_{\beta}(X)}

\newcommand{\perf}{\operatorname{Perf}(X)}
\newcommand{\perfstack}{\boldsymbol{\operatorname{Perf}}(X)}
\newcommand{\cohbetassstack}{\boldsymbol{\operatorname{Coh}}^{ss}_{\beta}(X)}
\newcommand{\cohflat}{\boldsymbol{\operatorname{Coh}}}

\newcommand{\framing}{\mathcal{V}}
\newcommand{\linebundle}{\mathcal{L}}

\newcommand{\copairs}{\boldsymbol{\cop}^{\beta}(X,\mathcal{V})}
\newcommand{\pairs}{\boldsymbol{\operatorname{P}}^{\beta}(X,\mathcal{V})}
\newcommand{\pairslb}{\boldsymbol{\operatorname{P}}^{\beta}(X,\mathcal{L})}
\newcommand{\Quotbeta}{\operatorname{Quot}^{\beta}(X,\mathcal{V})}
\newcommand{\Quotbetaone}{\operatorname{Quot}^{\beta}(X,\mathcal{L}[1])}
\newcommand{\Quotbetard}{\operatorname{Quot}_{r,d}^{\beta}(X,\mathcal{V})}
\newcommand{\Quotbetaonerd}{\operatorname{Quot}_{r,d}^{\beta}(X,\mathcal{L}[1])}
\newcommand{\bradstack}{\mathcal{B}^{\sigma - ss}_{r,d}(X,\mathcal{V})}
\newcommand{\brad}{\operatorname{{B}}^{\sigma - ss}_{r,d}(X,\mathcal{V})}
\newcommand{\bradstacklb}{\mathcal{B}^{\sigma - ss}_{r,d}(X,\mathcal{L})}
\newcommand{\bradlb}{\operatorname{{B}}^{\sigma - ss}_{r,d}(X,\mathcal{L})}
\newcommand{\dQuot}{\boldsymbol{\operatorname{Quot}}^{\beta}_{S}(X,\mathcal{V})}

\newcommand{\flags}{\boldsymbol{\operatorname{Perf}}^{\dagger}(X,\mathcal{V})}
\newcommand{\coflags}{\boldsymbol{\operatorname{Perf}}^{\ddagger}(X,\mathcal{V}[1])}

\newcommand{\pitfzero}{\pi^{\operatorname{t.f.}}_{0}}

\newcommand{\waldhausen}{\mathcal{S}_{\bullet}}

\newcommand{\cop}{\reflectbox{\ensuremath{\operatorname{P}}}}

\newcommand{\coext}{\reflectbox{\ensuremath{\operatorname{E}}}}

\subsectionfont{\large \textbf}

\title{Quot spaces in tilted hearts and Hall algebra modules}
\author{Niccolò Giacomini}
\date{}

\begin{document}

\maketitle

\begin{abstract}
    \noindent We construct a two--sided categorical action of the Hall algebra of semistable coherent sheaves of fixed slope on a curve $X$ on the derived category of certain Quot--spaces in tilted hearts on $X$. Following the philosophy in \cite{dps}, the action is induced by correspondence stacks that parameterize extensions of such quotients by semistable sheaves. In the process, we compare different moduli spaces on $X$: Quot--spaces, Bradlow pairs, and stable pairs in the sense of \cite{dps}.
\end{abstract}

\bigskip

\tableofcontents

\section{Introduction}
\noindent\textbf{Cohomological Hall algebras. }\textit{Cohomological Hall algebras (\textit{CoHA}'s)} are associative algebras that are modelled on the cohomology of moduli stacks. The procedure is as follows. Let $k$ be a field and let $\mathcal{A}$ be a $k$-linear abelian category. We set $\underline{\mathcal{A}}$ to be the moduli stack of objects in ${\mathcal{A}}$ and we set $\underline{\mathcal{A}}^{\operatorname{ext}}$ to be the moduli stack parametrizing extensions in $\mathcal{A}$. With these data, we set up the convolution diagram
\begin{equation}\label{convdiag}
   \begin{tikzcd}
	& {\underline{\mathcal{A}}^{\operatorname{ext}} } &&& {0 \rightarrow A' \rightarrow A \rightarrow A'' \rightarrow 0} & \\
	{\underline{\mathcal{A}} \times \underline{\mathcal{A}}} && {\underline{\mathcal{A}}} & {(A', A'')} && A
	\arrow["{p}"', from=1-2, to=2-1]
	\arrow["{q}", from=1-2, to=2-3]
	\arrow[maps to, from=1-5, to=2-4]
	\arrow[maps to, from=1-5, to=2-6]
\end{tikzcd}
    \end{equation}
Let $H^{BM}_*(\underline{\mathcal{A}})$ denote the Borel--Moore homology --- more generally any motivic Borel--Moore cohomology theory --- of the stack $\underline{\mathcal{A}}$. If the map $p$ is lci and the map $q$ is proper, the above diagram induces a map
\begin{equation}\label{product}
    q_*p^* \colon H^{BM}_*(\underline{\mathcal{A}})\otimes H^{BM}_*(\underline{\mathcal{A}}) \longrightarrow H^{BM}_*(\underline{\mathcal{A}}) \ 
\end{equation}
This makes the space $H^{BM}_*(\underline{\mathcal{A}})$ into an associative algebra. We refer to such algebra as the \textit{CoHA} associated to the category $\mathcal{A}$, and we denote it as $\operatorname{HA}_{\mathcal{A}}$. Now, consider a stack $\underline{\mathcal{A}}^{\operatorname{st}} \rightarrow \underline{\mathcal{A}}$, e.g. by introducing a \textit{framing and/or stability condition} on the objects of $\mathcal{A}$. Then, replacing $A$ and $A''$ in \eqref{convdiag} with stable objects in $\underline{\mathcal{A}}^{\operatorname{st}}$ yields a correspondence
\begin{equation}\label{action}
    \begin{tikzcd}
        &  (\underline{\mathcal{A}}^{\operatorname{ext}})^{\operatorname{st}} \arrow[dl, "{p}"'] \arrow[dr, "{q}"] & \\
 \underline{\mathcal{A}} \times \underline{\mathcal{A}}^{\operatorname{st}} && \underline{\mathcal{A}}^{\operatorname{st}} \ 
    \end{tikzcd}
    \end{equation}
for a suitable stack of extensions $(\underline{\mathcal{A}}^{\operatorname{ext}})^{\operatorname{st}} \rightarrow \underline{\mathcal{A}}^{\operatorname{ext}}$. With a similar procedure as before, this endows the space $H^{BM}_*(\underline{\mathcal{A}}^{\operatorname{st}})$ with the structure of a representation for the \textit{CoHA} $\operatorname{HA}_{\mathcal{A}}$. The geometric property encoding the associativity of a \textit{CoHA} action is that of a \textit{Hecke pattern}, in the sense of Definition 6.1 in \cite{mmsv23}. This property consists in asking that stable objects in $\mathcal{A}$ be preserved under extensions by objects in $\mathcal{A}$ --- see Example \ref{ex : hilbS}. However, Hecke patterns are very rare to find in the geometric setting. The few known examples of this procedure include \cite{mmsv23, Minets_2020, neguţ2022heckecorrespondencessmoothmoduli, dps, hausel2025pwmathcalh2, diaconescuhallalgebrasonedimensional}, and they all involve the category of zero dimensional sheaves on a smooth (quasi) projective surface, or torsion sheaves on a curve \cite{marian2026cohomologyquotschemesmooth, kaushik2026cohomologyhyperquotschemescurves}. \newline

\noindent\textbf{Doubles. }One of the main challenges in the theory of Hall algebras is computing their doubles, as we now explain. Suppose that we have a left and a right action of an algebra $\operatorname{A}$ on the same vector space $V$
\begin{equation}\label{leftright}
    \operatorname{A} \curvearrowright V \curvearrowleft \operatorname{A} \ . 
\end{equation}
Denoting as $\operatorname{A}^{op}$ the opposite algebra, we can ask whether we can place an algebra structure on $\operatorname{DA} = \operatorname{A}\otimes \operatorname{A}^{op}$ such that we have a naturally induced action
\begin{equation}\label{double action}
    \operatorname{DA \curvearrowright} V \ .
\end{equation}
Whenever this is possible, we refer to the algebra $\operatorname{DA}$ as the double of $\operatorname{A}$. This situation arises naturally in the context of Hall algebra actions. The most classical example involves the \textit{CoHA} associated with the category of zero dimensional sheaves on a surface, as defined by \cite{kapranov2022cohomologicalhallalgebrasurface}, generalizing \cite{Minets_2020, ss20, Zhao_2020}. The original motivation for these actions were the correspondences defined by Nakajima \cite{nakajimaheisenbergalgebrahilbertschemes} and Grojnowski \cite{grojnowski1995instantonsaffinealgebrasi}, later generalized by Negut \cite{neguţ2022heckecorrespondencessmoothmoduli} and Mellit--Minets--Schiffmann--Vasserot \cite{mmsv23}. 

\begin{example}\label{ex : hilbS}
    Let $S$ be a smooth projective surface, and denote as $\operatorname{Hilb}^n(S)$ the Hilbert scheme of $n$ points on $S$. It is shown in Proposition 7.5 of \cite{mmsv23} that the Hilbert scheme\footnote{More precisely, the Hilbert stack $\mathfrak{Hilb}^n(S) \cong \operatorname{Hilb}^n(S) \times B\mathbb{G}_m$.} $\operatorname{Hilb}^n(S)$ is a 2--sided Hecke pattern for the category $\operatorname{Coh}_0(S)$ of zero dimensional sheaves on $S$. In fact, consider a short exact sequence of coherent sheaves on $S$
    \[
        0 \rightarrow J \rightarrow I \rightarrow G \rightarrow 0
    \]
    where $G \in \operatorname{Coh}_0(S)$. Then, $I \in \operatorname{Hilb}^n(S)$ implies that $J \in \operatorname{Hilb}^{n-1}(S)$, and viceversa. This allows to construct a diagram as in \eqref{action} in two ways by setting
   \[\begin{tikzcd}
	& {0 \rightarrow J \rightarrow I \rightarrow G \rightarrow 0} &&& { 0 \rightarrow J \rightarrow I \rightarrow G \rightarrow 0} & \\
	{(G,I)} && {J} & {(J,G)} && I
	\arrow[maps to, from=1-2, to=2-1]
	\arrow[maps to, from=1-2, to=2-3]
	\arrow[maps to, from=1-5, to=2-4]
	\arrow[maps to, from=1-5, to=2-6]
\end{tikzcd}\]
    Then, "pull--push" via these two diagrams induce a left and a right action
    \[
       \operatorname{HA}_{\operatorname{Coh}_0(S)} \curvearrowright  \bigoplus_n {H}^*(\operatorname{Hilb}^n(S))\curvearrowleft  \operatorname{HA}_{\operatorname{Coh}_0(S)} 
    \]
    as in Proposition 6.6 of \cite{mmsv23}. Note that operators coming from the left action decrease $n$, whereas operators coming from the right action increase $n$. In particular, when the surface $S$ has trivial canonical bundle it is shown in \cite{mmsv23} that $\operatorname{HA}_{\operatorname{Coh}_0(S)}$ is isomorphic to a certain quantum deformation $W^+(S)$ of the BPS Lie algebra of $S$ \cite{dhm}. Moreover, the algebra $W^+(S)$ has a \textit{double} $W(S)$ with a triangular decomposition
$$
W(S) = W^+(S) \otimes W^0(S) \otimes W^-(S) , \ \ W^-(S) \cong (W^+(S))^\text{op} 
$$
which respects the \textit{CoHA} in its positive part. 
\end{example}

\noindent In the classical case where $\mathcal{A}$ is a hereditary and finitary category, the associated Hall algebra essentially encodes all the information regarding the structure of extensions in $\mathcal{A}$. Moreover, in this case the theory of Drinfeld doubles provides an algebraic solution to the problem discussed in this Section \cite{schiffmann2009lectureshallalgebras}. Even more, if $\mathcal{A}$ is the category of finite dimensional nilpotent representations of a quiver $Q$ over a field, then the double $D{HA}_\mathcal{A}$ recovers the quantum group associated with $Q$, and its natural triangular decomposition. In the two dimensional case, one can understand preprojective \textit{CoHA}'s as algebra of creating operators on the cohomology (and K--theory) of Nakajima quiver varieties \cite{MR3322196}. It has further been shown that such \textit{CoHA} can be identified with the positive half of the Maulik--Okounkov Yangian \cite{botta2025okounkovsconjecturebpslie, schiffmann2024cohomologicalhallalgebrasquivers}. In general, \textit{CoHA}'s are expected to provide a geometric way of constructing \textit{positive halves of quantum groups} and their representations. However we lack of a systematic way of understanding their doubles. Thus, one usually tries to construct left and right representations as in \eqref{leftright}, and glue them to the action of a bigger algebra. \newline

\noindent Let $X$ be a smooth complex projective curve, and let $\beta$ be a rational number. The category $\cohbetass$ of coherent semistable sheaves of fixed slope $\beta$ on $X$ has homological dimension $1$, but is not finitary. In this paper, we discuss a left and right action of the \textit{categorified Hall algebra} attached to the category $\cohbetass$ on the derived category of certain Quot spaces over the curve, as we now explain.

\subsubsection*{Quot spaces in tilted hearts and categorified action } 
\noindent \textbf{Algebra. }In Section \ref{sec : algebra} we recall the construction of the categorified Hall algebra (\textit{CatHA}) in our particular setting, following \cite{Porta2019TwodimensionalCH} and \cite{dps}.  In order to carry out the subsequent constructions, we freely use the language of derived stacks. We consider the standard derived enhancement $\cohbetassstack$ of the stack of semistable sheaves of slope $\beta$ on $X$. This stack fits in a derived convolution diagram
\begin{equation}\label{eqn : conv for ss in intro}
\begin{tikzcd}
	& {(\cohbetassstack)^{\operatorname{ext}}} & \\
	{\cohbetassstack \otimes \cohbetassstack} && \cohbetassstack
	\arrow["p", from=1-2, to=2-1]
	\arrow["q"', from=1-2, to=2-3]
\end{tikzcd}
\end{equation}
where (underived) points in the stack ${(\cohbetassstack)^{\operatorname{ext}}}$ correspond to extensions of semistable sheaves of slope $\beta$. As observed by Dyckeroff--Vasserot \cite{Dyckerhoff_2019} and Diaconescu--Porta--Sala \cite{dps}, such diagram has to be understood as the shadow of a more complicated object. More precisely, we recall in Subsection \ref{subsec : waldhausen} the notion of \textit{2--Segal space}, which is a simplicial stack encoding the higher combinatorics of the Hall product. Since the map $p$ is smooth and the map $q$ is proper we can define the composition 
\[
    q_*p^* \colon \operatorname{Coh}^{\operatorname{b}}(\cohbetassstack) \otimes \operatorname{Coh}^{\operatorname{b}}(\cohbetassstack) \longrightarrow \operatorname{Coh}^{\operatorname{b}}(\cohbetassstack)
\]
This endows the stable $\infty$--category $\operatorname{Coh}^{\operatorname{b}}(\cohbetassstack)$ with an $\mathbb{E}_1$--monoidal structure. We refer to the category $\operatorname{Coh}^{\operatorname{b}}(\cohbetassstack)$ together with this $\mathbb{E}_1$--monoidal structure as \textit{categorical Hall algebra of semistable sheaves of slope $\beta$ on $X$}. We can extract different \textit{CoHA}'s out of the CatHA. For example, the spaces
 \[ 
            G_0(\cohbetassstack) \hspace{1cm} \text{and} \hspace{1cm} H^{\text{BM}}_*(\cohbetassstack)
 \]
inherit the structure of unital associative algebras, recovering the construction of \cite{ss20, kapranov2022cohomologicalhallalgebrasurface, Zhao_2020} which agree with the formalism discussed in the previous Section. These algebra are the cohomological and K--theoretical Hall algebras associated with the category $\cohbetass$. \newline

\noindent\textbf{Module.} A natural way of considering {framed sheaves} is Grothendieck's Quot scheme, which parametrizes flat families of quotients of a fixed coherent sheaf $\framing$. The Hilbert scheme in Example \ref{ex : hilbS} is a particular case of Quot scheme, where $\framing = \mathcal{O}_S$. In order to set up an action diagram as in \eqref{action}, we study a generalization of such classical notion. In Section \ref{subsec : tilted heart abeta} we construct a family of hearts
\begin{equation}\label{tilted hearts}
    \abeta \subset D^b(\coh(X))
\end{equation}
by \textit{tilting} the standard t--structure with respect to a torsion pair. This is a standard procedure coming from the study of Bridgeland stabiliy conditions \cite{bridgeland2006stabilityconditionstriangulatedcategories} (see \cite{macri2019lecturesbridgelandstability} for a review). We denote as
    $$\operatorname{Quot}^\beta(X,\framing)$$ 
the punctual Quot space parametrizing quotients of an object $\framing$ in the heart $\abeta$. If $\framing =\linebundle[1]$, where $\linebundle$ is a line bundle on $X$ of degree smaller than $\beta$, then closed points in the scheme $\Quotbetaone$ can be described as short exact sequences
\begin{equation}\label{ses}
   0\rightarrow \linebundle \rightarrow F \rightarrow T \rightarrow 0
\end{equation}
of coherent sheaves where
\begin{itemize}
    \item the slope of every nonzero subsheaf of $F$ is smaller then or equal to $\beta$
    \item the slope of every nonzero quotient of $T$ is strictly bigger then $\beta$.
\end{itemize}
By using this explicit description, we relate the above Quot space with other relevant moduli spaces occurring in the literature.
\begin{enumerate}
    \item On the one hand the above description allows us to conclude that the Quot space $\Quotbetaone$ coincides with the classical truncation of the derived stack of \textit{$\beta$--stable $\linebundle$--pairs} $\pairslb$ in the sense of \cite{dps}, Definition III.4.10. See Proposition \ref{prop : PairsQuot}.

    \item On the other hand, we have \textit{semistable Bradlow pairs}, which arise from a geometric PDE on the curve $X$, called \textit{vortex equation} \cite{bradlow1, bradlow1993birationalequivalencesvortexmoduli}. In this context, the constraints on the short exact sequence \eqref{ses} come from the Kobayashi--Hitchin correspondence relating stability with the solvability of the aforementioned vortex equation. We prove in Proposition \ref{prop : BradlowQuot} that the above Quot space is an open subscheme in the (coarse) space of semistable Bradlow pairs. This investigation has overlappings with \cite{Rota_2021}, and \cite{10.4134/JKMS.J230331}.
\end{enumerate}

\noindent  The latter identification allows us to construct a categorical action of the Cat--Ha associated with the category $\cohbetass$ on $\operatorname{Coh}^{\operatorname{b}}(\Quotbetaone)$. We state our main result.

\begin{theorem}\label{theo 2}
     Let $\linebundle$ be a line bundle of slope smaller that $\beta$. Then the pro--$\infty$--category  \newline
\noindent$\operatorname{Coh}^{\operatorname{b}}(\Quotbetaone)$ carries the structure of a left and right categorical module over the CatHA of semistable sheaves of slope $\beta$. In particular 
      \[
            G_0(\Quotbetaone) \hspace{1cm} \text{and} \hspace{1cm} H_*^{BM}(\Quotbetaone)
       \]
       have the structure of a left and right module for $G_0(\cohbetassstack)$ and $H_*^{BM}(\cohbetassstack)$, respectively.
\end{theorem}

\noindent Our strategy to prove Theorem \ref{theo 2} is reminiscent of Example \ref{ex : hilbS}. We define in Section \ref{subsec : Construction of the left action} a space $\segal_1^l\Quotbetaone$ of extensions of quotients \eqref{ses} in the sense of Definition \ref{def : left extension}. This yields an action diagram 
\[
    \begin{tikzcd}
	& {\segal_1^l\Quotbetaone} & \\
	{\cohbetassstack \times \Quotbetaone} && \Quotbetaone 
	\arrow["{u^l_1 \times \omega_1}", from=1-2, to=2-1]
	\arrow["{\omega_0}"', from=1-2, to=2-3]
\end{tikzcd}
\]
and the left action is given by the pull--push operation
\begin{equation}\label{intro left action}
    (\omega_0)_*(u^l_1 \times \omega_1)^* \colon \operatorname{Coh^b}(\cohbetassstack) \otimes \coh^{\operatorname{b}}(\Quotbetaone) \longrightarrow \coh^{\operatorname{b}}(\Quotbetaone) 
\end{equation}
The analogous construction for the right action is presented in Section \ref{subsec : construction of the right action}. In order for the function\eqref{intro left action} to be well--defined we need to prove that the morphisms $\omega_0$ is proper and $u^l_1 \times \omega_1$ is derived lci. The two main technical points in the proof of Theorem \ref{theo 2} are Lemma \ref{lemma : aux rpas} and Lemma \ref{lemma : locally rpas in right action}, where we prove that the morphism $\omega_0$ and its right counterpart are proper. The proofs of such Lemmas rely in a substantial way on the properties of the tilted hearts \eqref{tilted hearts}, as mentioned in Remark \ref{rmk : coh ss and topology} and in the final point \ref{itm : coh ss topo prop} of Section \ref{subsec : construction of the right action}.

%\noindent As we mentioned above, the key geometric property behind the proof of Theorem \ref{theo 2} is that of a 2--sided Hecke pattern. In the context of this paper, we have to rely on the general concept of \textit{Segal Hecke patterns} in the sense of Definition II.3.9 of \cite{dps}. More precisely, the higher combinatorial structure of the left and right actions is encoded in the \textit{relative 2--Segal spaces} \eqref{eqn : relative simp pairs} and \eqref{eqn : relative simplicial copairs}, respectively. 

\subsection{Relation to other work}
Diaconescu--Porta--Sala constructed an action of the categorical Hall algebra associated with the category of zero dimensional sheaves on a surface $S$ on the derived category of the stack of Pandharipande--Thomas stable pairs on $S$ \cite{Pandharipande_2009} --- see Corollary 4.40 in \cite{dps}. In fact, it is shown in Proposition III.4.11 of \cite{dps} that Pandharipande--Thomas stable pairs can be described as stable pairs in the sense of Definition III.4.10 in \cite{dps}. On the other hand, Bridgeland \cite{bridgeland2020hallalgebrascurvecountinginvariants} related Pandharipande--Thomas stable pairs with Quot spaces in a tilted heart to derive wall--crossing formulas for Donaldson--Thomas invariants and Pandharipande--Thomas invariants. In the case of curves, Bradlow pairs have been related to Quot schemes in a tilted heart by different authors \cite{Rota_2021} \cite{10.4134/JKMS.J230331}. The main novelty in the present paper are the further connections between Quot spaces in tilted hearts, Bradlow pairs, and stable pairs on a curve. In turn, these connections are essential to the proof of Theorem \ref{theo 2}.

\subsection{Acknowledgements}
First, I want to thank my advisor Andrei Negut for his continued support and effective advice. I wish to thank Francesco Sala for sharing valuable insights about tilting and stable pairs. I thank Mauro Porta for his feedback and discussions. I further thank Ana Pavlakovic for introducing me to the notion of Bradlow pair. Moreover, I thank Tommaso Botta, Ben Davison, Dragos Fratila, Shivang Jidal, Archi Kaushik, Aliaksandra Novik, Olivier Schiffmann and 
Sebastian Schlegel Mejia for many interesting discussions and suggestions.

\subsection{Notations and conventions}
We make a list of notations and conventions that are used in this paper. We refer to \cite{lurie2008highertopostheory} for background material. We only work over the field $\BC$ of complex numbers.

\begin{itemize}
\item We denote as $\mathcal{S}$ the $\infty$--category of spaces
\item We denote as $\operatorname{Cat_{\infty}}$ the $\infty$--category of $\infty$--categories.
\item We denote as $\operatorname{Cat^{st}_{\infty}}$ the $\infty$--category of stable $\infty$--categories with exact functors between them.
\item Let $\mathcal{C}\in\operatorname{Cat_{\infty}}$. We denote as $\mathcal{C}^{\simeq}$ the maximal infinity groupoid contained in $\mathcal{C}$.
\item We denote as $\operatorname{dAff}$ the $\infty$--category of derived affine schemes.
\item We denote as $\operatorname{dSt}$ the $\infty$--category of derived stacks. This is the hypercompletion of the $\infty$ topos of sheaves on the site $\operatorname{dAff}$ with the Grothendieck topology induced by the étale  topology. 
\item We denote as $\operatorname{dGeom}$ the $\infty$--category of derived geometric stacks as in \cite{porta2016higheranalyticstacksgaga}.
\item Let $\operatorname{X} \in \operatorname{dSt}$. We denote as $\operatorname{Perf(X)}$ the $\infty$--category of perfect complexes on $\operatorname{X}$.
\end{itemize}

\section{Algebra}\label{sec : algebra}
Let $X$ be a smooth complex projective curve, and fix a number $\beta \in \BQ \sqcup\{+\infty\}$. In this Section we recall the construction of the Categorical Hall Algebras associated with the categories $\coh(X)$ of coherent sheaves and $\coh_{\beta}^{ss}(X)$ of semistable sheaves of slope $\beta$ on $X$, as defined by Diaconescu--Porta--Sala -- see Sections II.1.4 and II.1.5 in \cite{dps}.

\subsection{Introduction to the Section}\label{subsec : intro sec 2}

We recall the definition of the Cohomological Hall Algebra associated with the category $\cohbetass$. The stack $\underline{\coh}(X)$ parametrizes flat families of coherent sheaves on $X$, and it is a smooth algebraic stack locally of finite type over $\operatorname{Spec}(\BC)$. As we mentioned in the introduction, the \textit{CoHA} multiplication is given by the convolution diagram
\begin{equation}\label{eqn : coha diagram}
\begin{tikzcd}
	& {(\underline{\coh}(X))^{\operatorname{ext}}} & \\
	{\underline{\coh}(X) \otimes \underline{\coh}(X)} && \underline{\coh}(X)
	\arrow["p", from=1-2, to=2-1]
	\arrow["q"', from=1-2, to=2-3]
\end{tikzcd}
\end{equation}
where the stack ${(\underline{\coh}(X))^{\operatorname{ext}}}$ parametrizes extensions $0\rightarrow F' \rightarrow F \rightarrow F'' \rightarrow 0$ of coherent sheaves, and
\[
    p \colon 0\rightarrow F' \rightarrow F \rightarrow F'' \rightarrow 0 \mapsto (F',F'') \quad q \colon 0\rightarrow F' \rightarrow F \rightarrow F'' \rightarrow 0 \mapsto F
\]
Since the category $\coh(X)$ has homological dimension $1$, the map $p$ is smooth. It is also easy to see that the map $q$ is proper. In order to define a convolution product as in \eqref{product}, the authors of \cite{dps} consider the Borel--Moore homology --- or, more generally, any motivic Borel--Moore homology theory as developped by A. Khan \cite{khan1, khan2019virtualfundamentalclassesderived}  --- of the above correspondence. We are interested in the subcategory $\cohbetass  \subset \coh(X)$ of semistable sheaves of fixed slope $\beta$, for a fixed $\beta$. This category is closed under extensions. Moreover if $F\in \cohbetass$, then every subsheaf $F' \subset F$ of slope $\beta$ is semistable. Similarly, any quotient $F \twoheadrightarrow F''$ of slope $\beta$ is semistable. It follows that the above construction extends \textit{verbatim} to the open substack $\underline{\coh}_{\beta}^{ss}(X) \hookrightarrow \underline{\coh}(X)$
thus delivering the \textit{CoHA} of semistable sheaves of fixed slope. Porta--Sala \cite{Porta2019TwodimensionalCH} provided a categorification of this construction. That is, an $\mathbb{E}_1$--monoidal structure on the bounded derived category of the stack $\underline{\coh}(X)$ which recovers the above \textit{CoHA} under decategorification. In order to get such a structure, we need a suitable enhancement of the convolution diagram \eqref{convdiag}, which is provided by the \textit{Waldhausen construction} in Section \ref{subsec : waldhausen}. We explain the categorification to the case of semistable sheaves in Subsection \ref{subsec : Categorical Hall algebra of semistable sheaves of fixed slope}.

\subsection{Derived stacks of objects in hearts of t--structures }\label{subsec : Derived stacks of objects in hearts of t--structures }
    In order to carry out the constructions in the following Sections, we introduce derived and more general versions of the classical stacks mentioned in the previous Subsection. Let $\perfstack$ be the \textit{derived moduli stack of perfect complexes} on $X$. This is defined as the functor
    \begin{equation}\label{eqn : perf functor}
\perfstack\ \colon \operatorname{dAff}^{\operatorname{op}} \longrightarrow \mathcal{S}
    \end{equation}    
    which sends a derived affine scheme $S$ to the maximal $\infty$--grupoid $\operatorname{Perf}(X\times S)^{\simeq}$ contained in the stable $\infty$--category $\operatorname{Perf}(X\times S)$ of perfect complexes of quasi--coherent sheaves on $X\times S$ which have proper support over $S$. See Construction 2.5 in \cite{Porta2019TwodimensionalCH}, and example II.2.21 in \cite{dps}. This is the derived version of the stack of perfect complexes on $X$. \newline

    \noindent\textbf{Flatness and openness of t--structures. }
    In order to define the derived stack $\cohflat(X)$ we need to introduce the concept of flatness. We discuss a general notion of flatness with respect to a t--structure, as this will be useful in the following Sections. In the following we denote as $X_s$ the fiber of a morphism $X \times S \rightarrow S$.
    \begin{definition}
        Let $S$ be an affine derived scheme. A \textit{family of t--structures on }$X\times S$ is a collection $\{\tau_s\}_{s\in S}$, where $\tau_s$ is a t--structure on the fiber $X_{s}$.
    \end{definition}

    \noindent As an example, a t--structure on $\perf$ induces a \textit{constant} t--structure $\underline{\tau}$ on $X\times S$.

    \begin{definition}
        Let $\{\tau_s\}_{s\in S}$ be a family of t--structures on $D(X\times S)$. Let $s\in S$ be the image of a point $t\in T$ under a morphism $T \rightarrow S$ in $\operatorname{dAff}_{\BC}$. denote by $(\mathcal{A}_{qc})_t$ the heart of the t--structure induced on $\operatorname{Perf}(X_t)$ by base change. We say that $E\in \operatorname{Perf}(X_T)$ is \textit{flat with respect to} $\{\tau_s\}_{s\in S}$ if $E_t\in (\mathcal{A}_{qc})_t$ for all $t\in T$.
    \end{definition}

    \noindent Let $\tau$ be a t--structure on $\perf$. Then we have a subfunctor\footnote{This is analogous to Construction 2.5 in \cite{Porta2019TwodimensionalCH}.} of \eqref{eqn : perf functor}
    \begin{equation}\label{eqn : flatstack def}
        \cohflat(X,\tau) \longrightarrow \perfstack
    \end{equation}
   obtained by sending $S\in \operatorname{dAff}$ to the full subspace $\coh_S(X\times S)^{\simeq}$ spanned by families of coherent sheaves which are flat with respect to the constant t--structure $\underline{\tau}$. If we consider the standard t--structure $\tau_{\operatorname{std}}$, flatness reduces to the usual notion and we get a derived stack $\cohflat(X)$ of coherent sheaves on $X$ as in \cite{Porta2019TwodimensionalCH}. More precisely, the stack $\cohflat(X,\tau)$ is just the standard derived enhancement of the classical stack of families of $\tau$--flat sheaves defined in \cite{Bayer_2021}.
    
    \begin{proposition}
        The classical truncation $\prescript{cl}{}{\cohflat(X)}$ coincides with the usual stack $\underline{\coh}(X)$ of coherent sheaves on $X$.
    \end{proposition}
    \begin{proof}
        Let $S$ be an underived affine scheme. Then, the $S$--points in $\cohflat(X)$ are coherent sheaves on $X\times S$, flat and properly supported over $S$.
    \end{proof}

    \noindent We have an important definition.

    \begin{definition}
        A fiberwise family of t--structures $\underline{\tau}$ on $D(X \times S)$ \textit{universally satisfies openness of flatness} if for every derived affine scheme $S$ with a morphism $T\rightarrow S$, and every $T$--perfect object $E\in D(X_T)$, the set
        \[
            \{E\in D(X_T)\ \big| \ E_t \in (\mathcal{A}_{qc})_t)        
       \]
       is open.
    \end{definition}

    \noindent The following is going to be useful in many parts of the present paper. See Proposition II.2.56 in \cite{dps} for a non commutative (more general) version.
    
    \begin{proposition}\label{prop : geom locally finite presentation}
        The t--structure $\tau$ universally satisfies openness of flatness if and only if the morphism \eqref{eqn : flatstack def} is representable by Zariski open embeddings. In particular, if $\tau$ universally satisfies openness of flatness, the stack $\cohflat(X,\tau)$ is a geometric derived stack locally of finite presentation. 
    \end{proposition}

\subsection{Waldhausen construction}\label{subsec : waldhausen}
    We recall the construction of the categoricall Hall algebra associated with the stack $\cohflat(X)$ from \cite{Porta2019TwodimensionalCH}, building on \cite{Dyckerhoff_2019} and \cite{gaitsgory2017study}. We define the simplicial derived stacks encoding associativity of the Hall product. Let 
    \[
        \operatorname{T} := \operatorname{Hom}_{\Delta}([1] ,-) \colon \Delta \longrightarrow \operatorname{Cat}_{\infty}
    \]
    where $\Delta$ is the simplicial category. We denote $\operatorname{T}_n = \operatorname{T}([n])$.
    
\begin{definition}
    Let $\mathcal{C}$ be a $\BC$--linear stable $\infty$--category. We define $\mathcal{S}_n\mathcal{C}$ to be the full subcategory of $\operatorname{Fun}(\operatorname{T}_n,\mathcal{C})$ spanned by those functors $F$ satisfying the following assumptions.
    \begin{enumerate}
        \item $F(i,i) \simeq 0$ for every $0\leq i \leq n$.
        \item For every $0\leq i,j \leq n-1, \ i\leq j-1$, the square
        \[\begin{tikzcd}
	{F(i,j)} & {F(i+1,j)} \\
	{F(i,j+1)} & {F(i+1,j+1)}
	\arrow[from=1-1, to=1-2]
	\arrow[from=1-2, to=2-2]
	\arrow[from=1-1, to=2-1]
	\arrow[from=2-1, to=2-2]
\end{tikzcd}\]
    is a pullback in $\mathcal{C}$.
    \end{enumerate}
    \end{definition}
 \noindent It is then straightforward to check that the above construction induces a simplicial object
   \begin{equation}\label{eqn : waldhausen C}
        \waldhausen\mathcal{C} \colon \Delta^{op} \longrightarrow \operatorname{Cat}_{\infty}
   \end{equation}
   Theorem 7.3.3 in \cite{Dyckerhoff_2019} states that \eqref{eqn : waldhausen C} is a \textit{2--Segal space}, in the sense of Definition 2.3.1 therein. The importance of this notion in the context of the present paper is expressed by the following Theorem. We denote by $\operatorname{Corr}^{\times}(\operatorname{dSt})$ the $(\infty,2)$ category of correspondences of derived stacks, equipped with the symmetric monoidal structure induced by the cartesian product on $\operatorname{dSt}$ --- see \cite{gaitsgory2017study} $\S 7.2.1$.

    \begin{theorem}[Theorem 1.1 in \cite{godicke2024infty}]\label{thm : segal monoid}
        There is an equivalence of $\infty$--categories
        \[
            \operatorname{2--Segal}(\operatorname{dSt}) \longrightarrow \operatorname{Mon}_{\mathbb{E}_1}(\operatorname{Corr}^{\times}(\operatorname{dSt})), 
        \]
       where the left hand side is the category of 2--Segal derived stacks, and the right hand side is the category of $\mathbb{E}_1$--monoidal objects in the category $\operatorname{Corr}^{\times}(\operatorname{dSt})$.
    \end{theorem}

    \noindent \textbf{The stack $\waldhausen\cohflat(X,\tau)$. }This is the combinatorial data encoding the algebra structure. Next, we transfer these data to $\cohflat(X)$. We have    
    \[
        \waldhausen\perfstack\ \colon \Delta \longrightarrow \operatorname{dSt}
    \]
    by applying the {Waldhausen construction} to $\perfstack$. This is a simplicial derived stack such that 
    \[
    \segal_0\perfstack \cong \operatorname{Spec}(\BC), \ \segal_1\perfstack\cong \perfstack,\ \segal_2\perfstack \cong \perfstack^{\operatorname{ext}}\ ,
    \]
    where $\perfstack^{\operatorname{ext}}$ parametrizes extensions in $\perf$. The stack $\waldhausen\perfstack$ is naturally a { 2--Segal stack}. \newline
    
    \noindent Let $\boldsymbol{\operatorname{T}}$ be a derived stack equipped with a morphism
     \[
        \boldsymbol{\operatorname{T}} \longrightarrow \perfstack
     \]
     We can define a simplicial stack $\waldhausen\boldsymbol{\operatorname{T}}$ associated with the stack $\boldsymbol{\operatorname{T}}$ by setting

     \[\begin{tikzcd}
	{\segal_n\boldsymbol{\operatorname{T}}} & {\segal_n\perfstack} \\
	\boldsymbol{\operatorname{T}} & \perfstack
	\arrow[from=1-1, to=1-2]
	\arrow[from=1-1, to=2-1]
	\arrow[from=1-2, to=2-2]
	\arrow[""{name=0, anchor=center, inner sep=0}, from=2-1, to=2-2]
	\arrow["\lrcorner"{anchor=center, pos=0.125}, draw=none, from=1-1, to=0]
    \end{tikzcd}\]

    \noindent In particular, we can carry out this construction with respect to the embedding \eqref{eqn : flatstack def}. The following is a direct application of Proposition II.3.6 in \cite{dps}.
    
    \begin{proposition}\label{prop : coh is segal}
        The simplicial stack $\waldhausen\cohflat(X,\tau)$ is a 2--Segal stack.
    \end{proposition}

\subsection{Categorical Hall algebra of semistable sheaves of fixed slope}\label{subsec : Categorical Hall algebra of semistable sheaves of fixed slope}
     Fix a number $\beta \in \BQ\sqcup\{+\infty\}$. We have a substack
    \begin{equation}\label{eqn : cohssbeta embedding}
        \cohbetassstack \hookrightarrow \cohflat (X)
    \end{equation}
    of semistable coherent sheaves of slope $\beta$. This embedding is open thanks to Lemma 21.12 in \cite{Bayer_2021}. Proceeeding as in Subsection \ref{subsec : waldhausen} we can define the derived simplicial stack $\waldhausen\cohbetassstack$. In particular, under the isomorphism $\segal_2\cohbetassstack \cong (\cohbetassstack)^{\operatorname{ext}}$ we recover the convolution diagram \eqref{eqn : conv for ss in intro} from the introduction
    \[\begin{tikzcd}
	& {(\cohbetassstack)^{\operatorname{ext}}} & \\
	{\cohbetassstack \otimes \cohbetassstack} && \cohbetassstack
	\arrow["\partial_0 \times \partial_2", from=1-2, to=2-1]
	\arrow["\partial_1"', from=1-2, to=2-3]
\end{tikzcd}\]
where the map $\partial_i$ is induced by the $i$-th face map. Explicitely, we have
\[
\partial_0 \times\partial_2 \colon 0\rightarrow F'\rightarrow F \rightarrow F'' \rightarrow 0 \mapsto (F', F'') \quad \text{and} \quad \partial_1  \colon 0\rightarrow F'\rightarrow F \rightarrow F'' \rightarrow 0 \mapsto F\ .
\]
We extract a stable $\infty$--category with a $\mathbb{E}_1$--monoidal structure out of the above correspondence. 
    \begin{definition}
            Let $f \colon \operatorname{X} \rightarrow Y$ be a morphism in $\operatorname{dGeom}$.
        \begin{enumerate}
            \item The morphism $f$ is \textit{derived lci} if for any $\operatorname{Z} \in \operatorname{dGeom^{qc}}$, the pullback $\operatorname{X\times_{Y}Z}$ is a quasi--compact derived geometric stack and the map $\operatorname{X\times_{Y}Z} \rightarrow \operatorname{Z}$ is derived lci.
            \item The morphism $f$ is \textit{locally rpas} if for every connected component $\operatorname{X_0}$ of $\operatorname{X}$, the composite map $\operatorname{X_0}\rightarrow \operatorname{Y}$ is representable by proper algebraic spaces.
            \item The morphism $f$ is \textit{finitely connected} if for any connected component $\operatorname{Y_0}$ of $\operatorname{Y}$, the stack $f^{-1}(\operatorname{Y_0}) = \operatorname{X\times_{Y}Y_0}$ has finitely many connected components.
        \end{enumerate}
    \end{definition}
    \noindent We can define a refined pullback $f^*$ for any derived lci map, and we can define pushforward $f_*$ for any locally rpas map. Let
    \[
      \operatorname{dGeom}^{\operatorname{qc}} \hookrightarrow \operatorname{dSt}  
    \]
    be the subcategory of quasi--compact geometric derived stack, and consider the $(\infty,2)$ category 
    \[
        \operatorname{Corr}^{\times}(\operatorname{dGeom}^{\operatorname{qc}})_{\operatorname{rpas,lci}} \hookrightarrow  \operatorname{Corr}^{\times}(\operatorname{dSt})
    \]
    of correspondences of quasi--compact geometric derived stacks where the horizontal morphisms are ind--derived lci and the vertical morphisms are locally rpas. Building on Gaitsgory--Rozenblyum's work \cite{gaitsgory2017study}, Porta--Sala --- see the appendix of \cite{Porta2019TwodimensionalCH} --- defined a right--lax monoidal functor
    \begin{equation}\label{eqn : coh for qc}\coh^{\operatorname{b}} \colon  \operatorname{Corr}^{\times}(\operatorname{dGeom}^{\operatorname{qc}})_{\operatorname{rpas,lci}} \longrightarrow \operatorname{Cat}^{\operatorname{st}}_{\infty}
    \end{equation}
    \noindent In particular $\mathbb{E}_1$--monoid objects are preserved under this functor. The following Theorem is an application of Corollary II.3.7 in \cite{dps}. See also Corollary II.4.14 in \cite{dps}.

    \begin{theorem}\label{thm : cathabeta}
    The following conditions are satisfied.
    \begin{enumerate}
        \item The stack $\cohbetassstack$ is quasi--geometric and locally of finite presentation. 

        \item\label{itm : catbeta segal} The simplicial stack $\waldhausen\cohbetassstack$ is a 2--Segal space.
        
        \item\label{itm : catmeta smooth} The map 
            \begin{equation}
                \partial_0 \times \partial_2\ \colon \ \segal_2\cohbetassstack \longrightarrow \cohbetassstack \times \cohbetassstack
            \end{equation}
            is quasi--compact, finitely connected and derived lci.
        \item\label{itm : catbeta proper} The map
        \begin{equation}
            \partial_1\ \colon \ \segal_2\cohbetassstack \longrightarrow \cohbetassstack
        \end{equation}
        is locally rpas.
        \end{enumerate}
        Then, $\coh^{\operatorname{b}}(\cohbetassstack)$ has the structure of an $\mathbb{E}_1$--monoidal stable pro--$\infty$--category, whose underlying tensor product is given by the composition
        \[
            \coh^{\operatorname{b}}(\cohbetassstack)\times \coh^{\operatorname{b}}(\cohbetassstack) \xrightarrow{(\partial_{1})_*\ \circ \ (\partial_0 \times\partial_2)^*} \coh^{\operatorname{b}}(\cohbetassstack)\ .
        \]
        In particular, 
        \[ 
            G_0(\cohbetassstack) \hspace{1cm} \text{and} \hspace{1cm} H^{\text{BM}}_*(\cohbetassstack)
        \]
        are unital associative algebras.
    \end{theorem}
    \begin{proof}
        Proposition \ref{prop : geom locally finite presentation} tells us that the stack $\cohflat(X)$ is quasi--geometric and locally of finite presentation. Since embedding \eqref{eqn : cohssbeta embedding} is open, we deduce that the same holds for $\cohbetassstack$. Condition  \ref{itm : catbeta segal}, follows from Corollary I.5.7 in \cite{dps}. The connected components of the stack $\cohbetassstack$ are quasi--compact  --- see \cite{Huybrechts_Lehn_2010}  ---, so Theorem \eqref{thm : segal monoid} tells us that the 2--Segal stack $\waldhausen\cohbetassstack$ defines an $\mathbb{E}_1$--monoidal object in $\operatorname{Corr}^\times(\operatorname{dGeom}^{\operatorname{qc}})$. Let 
        \begin{equation}\label{eqn : ses in the proof of catha}
            0 \rightarrow F' \rightarrow F \rightarrow F'' \rightarrow 0
        \end{equation}
        be a short exact sequence of coherent sheaves. The condition $F', F'' \in \cohbetass$ implies $F \in \cohbetass$. In particular the diagram
        \[\begin{tikzcd}
	{\segal_2\cohbetassstack} & {\segal_2\cohflat(X)} \\
	{\cohbetassstack \times \cohbetassstack} & {\cohflat(X) \times \cohflat(X)}
	\arrow[from=1-1, to=1-2]
	\arrow["{\partial_0 \times \partial_2}"', from=1-1, to=2-1]
	\arrow["{\partial_0\times \partial_2}", from=1-2, to=2-2]
	\arrow[from=2-1, to=2-2]
\end{tikzcd}\]
is a pullback. This implies condition \ref{itm : catmeta smooth}. Moreover, if all the sheaves in \eqref{eqn : ses in the proof of catha} have slope $\beta$, then $F\in \cohbetass$ implies $F',F''\in \cohbetass$. Then the square
\[\begin{tikzcd}
	{\segal_2\cohbetassstack} & {\segal_2\cohflat_{\beta}(X)} \\
	\cohbetassstack & \cohflat_{\beta}(X)
	\arrow[from=1-1, to=1-2]
	\arrow["{\partial_1}"', from=1-1, to=2-1]
	\arrow["{\partial_1}", from=1-2, to=2-2]
	\arrow[from=2-1, to=2-2]
\end{tikzcd}\]
is a pullback, where $\cohflat_{\beta}(X)$ denotes the open and closed substack of $\cohflat(X)$ parametrizing all sheaves of slope $\beta$. Condition \ref{itm : catbeta proper} follows. Then, the stack $\waldhausen\cohbetassstack$ defines a $\mathbb{E}_1$--monoidal object in $\operatorname{Corr}^{\times}(\operatorname{dGeom}^{\operatorname{qc}})_{\operatorname{rpas,lci}}$, and applying the lax functor \eqref{eqn : coh for qc} completes the proof.
    \end{proof}

\section{Module}
For a given object $\framing$ in $\perf$ and a real number $\beta\in \BR$, we define a derived moduli stack $\pairs$ of $\beta$--stable $\framing$--pairs on $X$. When $\beta \in \BQ$, the derived category of the stack $\pairs$ will support an action of the algebra defined in the previous Section. Moreover, we compare the stack $\pairs$ to other moduli spaces known in the literature.

\subsection{Torsion pairs}
We follow \cite{Bayer_2021} and \cite{Rota_2021}. Recall the definition of torsion pair.

\begin{definition}\label{def : torsionpair}
    Let $\mathcal{A}$ be an abelian category. A \textit{torsion pair} on $\mathcal{A}$ is a pair $(\mathcal{F},\mathcal{T})$ of full subcategories of $\mathcal{A}$ such that
    \begin{enumerate}
        \item For every $T \in \mathcal{T}$ and $F\in \mathcal{F}$, we have $\operatorname{Hom}_{\mathcal{A}}(T,F) =\{0\}$
        \item Every object $A\in \mathcal{A}$ fits into an exact sequence
        $$
         0\rightarrow T \rightarrow A \rightarrow F \rightarrow 0
        $$
        for some $T \in \mathcal{T}$ and $F\in \mathcal{F}$.
    \end{enumerate}
    We refer to $\mathcal{F}$ as the\textit{ torsion free} part, and to $\mathcal{T}$ as the \textit{torsion part}.
\end{definition}

\noindent We list some useful properties of torsion pairs for future reference.

\begin{lemma}\label{lemma : torsionproperties}
    Consider a torsion pair $(\mathcal{F},\mathcal{T})$ in an abelian category $\mathcal{A}$. Let
    $$
        A' \rightarrow A \rightarrow A'' 
    $$
    be a fibered sequence in $\mathcal{A}$. Then the following hold.
    
    \begin{enumerate}
        \item $A\in \mathcal{F} \Rightarrow A'\in \mathcal{F}$
        \item $A\in \mathcal{T} \Rightarrow A''\in \mathcal{T}$
        \item $A'$ and $A''$ are in $\mathcal{F} $, then $A \in \mathcal{F}$
        \item $A'$ and $A''$ are in $\mathcal{T} $, then $A \in \mathcal{T}$.
    \end{enumerate}
\end{lemma}

\noindent\textbf{Tilting.} Let $\mathcal{D}=D^b(\mathcal{A)}$ be the derived category of bounded complexes in $\mathcal{A}$. The \textit{tilting} of $\mathcal{A}$ with respect to a torsion pair $(\mathcal{F},\mathcal{T})$ is the smallest extesion--closed subcategory of $\mathcal{D}$ containing both $\mathcal{T}$ and $\mathcal{F}[1]$. We denote it as $\mathcal{A}^*=\langle\mathcal{F}[1],\mathcal{T}\rangle$, and it can be explicitely described as
\begin{equation}
    \mathcal{A}^{*} = \left\{G\in \mathcal{D}\Big|  \mathcal{H}^i(G) =0 \text{ if } i\neq 0,-1 \ ; \ \mathcal{H}^0(G) \in \mathcal{T} ; \ \mathcal{H}^{-1}(G) \in \mathcal{F}  
    \right\} \ ,
\end{equation}
where $\mathcal{H}^i$ denotes the cohomology with respect to the standard t--structure $\tau_{\operatorname{std}}$. Moreover, we can define a bounded t--structure $\tau^*$ on $\mathcal{D}$ as
\[
    \mathcal{D}^{\leq 0} = \{G \in \mathcal{D} \Big| \mathcal{H}^i(G)=0 \text{ for } i>0 \text{ and } \mathcal{H}^0(G)\in \mathcal{T}\}
\]
\[
    \mathcal{D}^{\geq 0} = \{G \in \mathcal{D} \Big| \mathcal{H}^i(G)=0 \text{ for } i<-1 \text{ and } \mathcal{H}^{-1}(G)\in \mathcal{F}\}
\]

\begin{proposition}[\cite{macri2019lecturesbridgelandstability}]
    The abelian category $\mathcal{A}^*$ is the heart of the t--structure $\tau^*$ on $\mathcal{D}$.
\end{proposition}

\noindent An object $G$ in the tilted heart $\mathcal{A}^*$ fits in a canonical short exact sequence
\begin{equation}\label{eqn : canonical ses tilting}
    0 \rightarrow \mathcal{H}^{-1}(G)[1] \rightarrow G \rightarrow \mathcal{H}^0(G) \rightarrow 0\ .
\end{equation}
Then, one can easily deduce that the subcategories
\[
    \mathcal{F}[1]=\{G\in \mathcal{A}^* \ \big|\ \mathcal{H}^0(G) =0 \} \hspace{1cm} \text{    and    } \hspace{1cm} \mathcal{T}=\{ G\in \mathcal{A}^* \ \big| \ \mathcal{H}^{-1}(G)=0\} 
\] 
form a torsion pair in the tilted heart $\mathcal{A}^*$. In particular, $\mathcal{F}[1]$ is the torsion part, and $\mathcal{T}$ is the torsion free part. 

\bigskip

\begin{definition}
    An abelian category $\mathcal{A}$ is \textit{noetherian} if for any object $A\in \mathcal{A}$ and for any ascending chain of subobjects in $\mathcal{A}$
    \[
        A_0\subset A_1 \subset \dots \subset A_i \subset \dots \subset A
    \]
    we have $A_i = A_j $ for $i,j$ big enough. 
\end{definition}

\subsection{The tilted heart $\abeta$ }\label{subsec : tilted heart abeta}
Let $X$ be a smooth projective curve defined over the field $\BC$ of complex numbers, and let $\beta\in \BR$. Consider the full subcategories of $\coh(X)$ given by
\begin{equation} \label{eqn : fbetatbeta}
\mathcal{F}^{\beta}=\{ G\in \operatorname{Coh}(X) \ |\ \mu(G')\leq \beta \text{ for all non-zero subsheaves } G'\subseteq G \} \ ,
\end{equation}
\[
\mathcal{T}^{\beta}=\{ G\in \operatorname{Coh}(X) \ |\ \mu(G'') > \beta \text{ for all quotients } G \twoheadrightarrow G''\neq G \}.
\]
\noindent The following is an exercise in \cite{jardim2025stabilityconditionscoherentsystems}.

\begin{lemma}\label{lemma : torsion pair proof}
     The pair $(\mathcal{F}^{\beta}, \mathcal{T}^{\beta})$ is a torsion pair in $\operatorname{Coh}(X)$ 
where $\mathcal{F}^{\beta}$ is the torsion free part, and $\mathcal{T}^{\beta}$ is the torsion part.
\end{lemma}
\begin{proof}
    We observe that
    \[
\mathcal{F}^{\beta}=\{ G\in \operatorname{Coh}(X) \ |\ \text{the slope of each semistable factor is } \leq \beta \} \ ,
\]
\[
\mathcal{T}^{\beta}=\{ G\in \operatorname{Coh}(X) \ |\ \text{the slope of each semistable factor is } > \beta \}
\]
and use the Harder--Narashiman filtration.
\end{proof}

\noindent  Tilting the standard heart $\coh(X)$ with respect to the above torsion pair yields another t-structure which we denote as $\tau^\beta$, whose heart is given by 
\begin{equation}
    \mathcal{A}^{\beta} = \left\{E\in \perf \Big|  \mathcal{H}^i(E) =0 \text{ if } i\neq 0,-1 \ ; \ \mathcal{H}^0(E) \in \mathcal{T}^{\beta} ; \ \mathcal{H}^{-1}(E) \in \mathcal{F}^{\beta}  
    \right\} \ .
\end{equation}
\noindent We denote as
\[
    \abetatf= \fbeta[1] \hspace{1cm} \abetator = \tbeta 
\]
the induced torsion pair.

\begin{lemma}\label{lemma : notherianity}
The following statements hold.
\begin{enumerate}
    \item  The  t--structure $\tau^{\beta}$ universally satisfies openness of flatness.
    \item If $\beta \in \BQ$ the tilted heart $\abeta$ is noetherian.
\end{enumerate}
\end{lemma}
\begin{proof}
The first statement follows from the general theory on Bridgeland stability conditions. See also the following Remark \ref{rmk : openness of flatness}. The noetherianity statement for rational $\beta$ is proven in Lemma 6.17 in \cite{macri2019lecturesbridgelandstability}. 
\end{proof}

\begin{remark}\label{rmk : noeth fails if beta is not rational}
    The tilting procedure usually destroys the noetherianity property. Indeed, if $\beta\in \BR - \BQ$, Noetherianity of the t--structure $\tau^\beta$ fails --- See Example 3.3 in \cite{Rota_2021}. A similar phenomenon has been observed recently in the context of coherent systems on curves --- see \cite{feyzbakhsh2025derivedcategorycoherentsystems} and \cite{jardim2025stabilityconditionscoherentsystems}. However it is still true that the t--structure $\tau^\beta$ universally satisfies openness of flatness, as we discuss in Remark \ref{rmk : openness of flatness}. 
\end{remark}
\bigskip

\noindent We end the Subsection with a remark which will be useful later on. 

\begin{lemma}\label{lemma : cohbetacirc}
    The inclusion
    \begin{equation}\label{eqn : cohbetacirc}
        \cohflat(X,\tau^\beta)^\circ \hookrightarrow \cohflat(X,\tau^{\beta})
    \end{equation}
    defined by the condition
    \[
        \mu(\mathcal{H}^{-1}(G))<\beta  
    \]
    is representable by Zariski open embeddings.
\end{lemma}

\begin{proof}
    One can give a tweaked definition of the torsion pair \eqref{eqn : fbetatbeta} by setting
\begin{equation}\label{eqn : tweaked tp}
    \widetilde{\mathcal{F}}^{\beta}=\{ G\in \operatorname{Coh}(X) \ |\ \mu(G')< \beta \text{ for all non-zero subsheaves } G'\subseteq G \} 
\end{equation}
\[
\widetilde{\mathcal{T}}^{\beta}=\{ G\in \operatorname{Coh}(X) \ |\ \mu(G'') \geq \beta \text{ for all quotients } G \twoheadrightarrow G''\neq G \}.
\]
This definition induces a torsion pair in $\coh(X)$ by the same argument as in Lemma \ref{lemma : torsion pair proof}. If we denote by $\widetilde{\tau}^\beta$ the t--structure whose heart is the tilting of $\coh(X)$ with respect to this torsion pair, we get a substack
\[
    \cohflat(X,\widetilde{\tau}^{\beta}) \longrightarrow \perfstack
\]
As before, this morphism is representable by Zariski--open embeddings. Now observe that the stack $\cohflat(X,\tau^\beta)^\circ$ coincides with the intersection
\[
    \cohflat(X,\tau^{\beta}) \cap \cohflat(X,\widetilde{\tau}^{\beta})\ ,
\]
and the claim follows.
\end{proof}

\begin{remark}
    The openness of the stack \eqref{eqn : cohbetacirc} in $\perfstack$ follows from Proposition 20.8 in \cite{Bayer_2021}. In particular, it is described in terms of \textit{slicing} associated to a Bridgeland stability condition.
\end{remark}

\subsection{Derived stack of $\beta$--stable $\framing$--pairs}\label{subsec : derived stack of pairs} 
This subsection is devoted to defining and describing the spaces that support the action in Theorem \ref{theo 2}. 

\bigskip

\noindent \textbf{Stacks of torsion and torsion free objects. }We can define a derived stack $\tbetastack$ parametrizing flat families of sheaves in $\tbeta$. Specifically, $\tbetastack$ is defined by the pullback square
\begin{equation}\label{eqn : tstd open}
\begin{tikzcd}
	\tbetastack & {\cohflat(X,\tau^\beta)} \\
	{\cohflat(X)} & \perfstack
	\arrow[from=1-1, to=1-2]
	\arrow[from=1-1, to=2-1]
	\arrow[from=1-2, to=2-2]
	\arrow[""{name=0, anchor=center, inner sep=0}, from=2-1, to=2-2]
	\arrow["\lrcorner"{anchor=center, pos=0.125}, draw=none, from=1-1, to=0] \ 
\end{tikzcd}
\end{equation}
where the bottom horizontal map and the right vertical maps are given by the embedding \eqref{eqn : flatstack def} for the t--structures $\tau_{\operatorname{std}}$ and $\tau^\beta$, respectively. Similarly, we define a derived stack $\fbetastack$ parametrizing flat families of sheaves in $\fbeta$. Let $[1] \colon \perfstack \rightarrow \perfstack$ be the morphism induced by the shift by $1$ in $\perf$. Then we can also define the stack $\fbetastack$ via the pullback
\begin{equation}
    \begin{tikzcd}
	\fbetastack && {\cohflat(X,\tau^\beta)} \\
	{\cohflat(X)} & \perfstack & \perfstack
	\arrow[from=1-1, to=1-3]
	\arrow[from=1-1, to=2-1]
	\arrow[from=1-3, to=2-3]
	\arrow[""{name=0, anchor=center, inner sep=0}, from=2-1, to=2-2]
	\arrow["{[1]}", from=2-2, to=2-3]
	\arrow["\lrcorner"{anchor=center, pos=0.125}, draw=none, from=1-1, to=0]
\end{tikzcd}
\end{equation}

\begin{proposition}\label{prop : open tp}
    The morphisms
    \begin{equation}\label{eqn : embeddings of tbeta and fbeta}
        \fbetastack \rightarrow \cohflat(X) \hspace{1cm} \text{and} \hspace{1cm} \tbetastack\rightarrow\cohflat(X)
    \end{equation}
    are representable by Zariski open embeddings.
\end{proposition}
\begin{proof}
    This statement follows from the general results of \cite{Bayer_2021} in the underived setting. Observe that the derived stack $\cohflat(X)$ has a decomposition into open and closed substacks
\[
    \cohflat(X) = \bigsqcup_{r,d} \cohflat_{r,d}(X)
\]
according to the topological type. We have a similar decomposition for the derived stacks $\tbetastack$ and $\fbetastack$. Let $S$ be a connected algebraic space. We want to show that the left vertical map in the squares
    \begin{equation}\label{eqn : cohflat is representable}
\begin{tikzcd}
	{S\times_{\cohflat(X)} \cohflat_{\operatorname{t.f.}, (r,d)}(X)} & \cohflat_{\operatorname{t.f.}, (r,d)}(X) & {S \times_{\cohflat(X)} \cohflat_{\operatorname{tor}, (r,d)}(X)} & \cohflat_{\operatorname{tor}, (r,d)}(X) \\
	S & {\cohflat_{r,d}(X)} & S & {\cohflat_{r,d}(X)}
	\arrow[from=1-1, to=1-2]
	\arrow[from=1-1, to=2-1]
	\arrow[from=1-2, to=2-2]
	\arrow[from=1-3, to=1-4]
	\arrow[from=1-3, to=2-3]
	\arrow[from=1-4, to=2-4]
	\arrow[from=2-1, to=2-2]
	\arrow[from=2-3, to=2-4]
\end{tikzcd}
    \end{equation}
are open embeddings, for any $r,d$. By the Yoneda lemma, the bottom horizontal map is identified by a flat family of coherent sheaves on $X$, parametrized by $S$. Such family is given by a sheaf $T\in \coh(X \times S)$. For any $s\in S$ we denote as $T_s\in \coh(X\times \{s\})$ the restriction. Then the pullbacks
\begin{equation}\label{eqn : S'}
    S\times_{\cohflat(X)} \cohflat_{\operatorname{t.f.}, (r,d)}(X) = \{s\in S \ \Big| \ T_s \in \fbeta \}
\end{equation}
\[
    S\times_{\cohflat(X)}\cohflat_{\operatorname{tor}, (r,d)}(X) = \{s\in S \ \Big| \ T_s \in \tbeta \}
\]
are algebraic spaces. Hence the right vertical maps in \eqref{eqn : cohflat is representable} are representable. We have to prove that such maps are open embeddings. It follows from the Corollary to Proposition 10 in \cite{shatz} --- see also Lemma 20.9 in \cite{Bayer_2021} --- that 
\begin{equation}\label{eqn : mu+}
     S \longrightarrow \BQ \cup \{+\infty\}, \quad
    	 s \mapsto \mu_+(T_s) =\max\{ \mu(T'_s) \ \Big|\  T'_S \text{ is a proper subsheaf of } T_s \}
\end{equation}
is a upper semi--continuous and constructible function, and 
\begin{equation}\label{eqn : mu-}
  S \longrightarrow \BQ \cup \{+\infty\}, \quad
    	 s \mapsto \mu_-(T_s) =\min\{ \mu(T''_s) \ \Big|\  T''_S \text{ is a proper quotient of } T_s \}  
\end{equation}
is lower semi--continuous constructible function. Observe that for any fixed topological type we have an equality $\mu_+^{-1}([-\infty , \beta]) = \mu_+^{-1}([-\infty , \beta + \epsilon))$, for a suitable $\epsilon > 0$. In particular,
\[
    S\times_{\cohflat(X)} \cohflat_{\operatorname{t.f.}, (r,d)}(X)=  \mu_+^{-1}([-\infty, \beta + \epsilon)) 
\]
and
\[
    S\times_{\cohflat(X)} \cohflat_{\operatorname{tor}, (r,d)}(X) = \mu_-^{-1}((\beta,+\infty])  \ 
\]
are open in $S$. Therefore 
\[
    S\times_{\cohflat(X)} \cohflat_{\operatorname{t.f.}, (r,d)}(X) = \mu_+^{-1}([-\infty, \beta]) 
\]
is open in $S$, as we had to prove. 

\end{proof}

\begin{remark}\label{rmk : openness of flatness}
    Proposition \ref{prop : open tp} says that the torsion pair $(\fbeta, \tbeta)$ is open in the sense of Definition II.2.58 in \cite{dps}. Since the standard t--structure $\tau_{\operatorname{std}}$ universally satisfies openness of flatness, Proposition II.2.52 in \cite{dps} implies that the t--structure $\tau^\beta$ universally satisfies openness of flatness for all $\beta \in \BR$.
\end{remark}

\noindent Repeating the construction, we get stacks
\begin{equation}\label{eqn : tbeta open}
   \tbetastack\cong \abetatfstack\rightarrow \cohflat(X,\tau^\beta) \hspace{0.5cm} \text{and} \hspace{0.5cm}  \fbetastack[1] \cong \abetatorstack\rightarrow \cohflat(X,\tau^\beta)
\end{equation}
where $\fbetastack[1]$ denotes the image of the stack $\fbetastack$ under the shift morhphism in $\perfstack$. 

\begin{corollary}
     The morphisms \eqref{eqn : tbeta open} are representable by Zariski--open embeddings.
\end{corollary}
\begin{proof}
    Since the t--structure $\tau^\beta$ universally satisfies openness of flatness, the embedding
    \[
        \cohflat(X,\tau^\beta) \longrightarrow \perfstack
    \]
    is open. Thus, the claim follows from Proposition \ref{prop : open tp}.
\end{proof}

\noindent\textbf{Stable pairs. }We introduce the notion of stable pair. Compare with Definition III.4.10 in \cite{dps}.

\begin{definition}
    Let $\framing$ be an object in $\perf$. A $\framing$--pair is a fiber sequence of the form $\operatorname{P}= [\framing\rightarrow F \rightarrow T]$. We say that $\operatorname{P}$ is a $\beta$--stable $\framing$--pair if 
    \begin{enumerate}
        \item $F[1]\in \abetator$ (equivalently $F \in \fbeta$), and
        \item $T\in \abetatf$(equivalently $T \in \tbeta)$.
    \end{enumerate}
    The topological type of a $\framing$--pair $\operatorname{P}= [\framing\rightarrow F \rightarrow T]$ is the Chern class of $F$.
\end{definition}

\noindent It follows immediatly from the definition that if $\framing$ is a coherent sheaf, then a $\beta$--stable $\framing$--pair is a short exact sequence in $\coh(X)$. We now define a stack $\pairs$ parametrizing stable pairs as in the above Definition, following Definition II.3.2 in \cite{dps}. Recall the Waldhausen construction $\waldhausen\perfstack$ from \eqref{eqn : waldhausen C}. Unravelling the definitions, we see that closed points in the stack $\segal_2\perfstack$ are diagrams of the form 
\[\begin{tikzcd}
	0 & \framing& F \\
	& 0 & T \\
	&& 0
	\arrow[from=1-1, to=1-2]
	\arrow[from=1-2, to=1-3]
	\arrow[from=1-2, to=2-2]
	\arrow[from=1-3, to=2-3]
	\arrow[from=2-2, to=2-3]
	\arrow[from=2-3, to=3-3]
\end{tikzcd}\]
where $\framing, F, T$ are objects in $\perf$ and the square is a pullback. In other words, closed points in $\segal_2\perfstack$ are fiber sequences $[\framing\rightarrow F \rightarrow T]$ in $\perf$. Under the natural identification $\segal_1\perfstack \cong \perfstack$, we also have natural maps
\begin{equation}\label{eqn : delta maps}
\partial_i \colon \segal_2\perfstack \longrightarrow \perfstack
\end{equation}
for $i=1,2,3$, given by
\begin{equation}\label{eqn : delta maps}
    \partial_0([\framing\rightarrow F \rightarrow T])  = T \quad \partial_1([\framing\rightarrow F \rightarrow T]) = F \quad \partial_2([\framing\rightarrow F \rightarrow T]) = \framing
\end{equation}

\begin{definition}\label{def : perf}
    Fix $\framing \in \perf$. We define the\textit{ stack of 2--flags} $\flags$ as the fiber product
    \[\begin{tikzcd}
	\flags & {\segal_2\perfstack} \\
	{Spec(\BC)} & \perfstack
	\arrow[from=1-1, to=1-2]
	\arrow[from=1-1, to=2-1]
	\arrow["{\partial_2}", from=1-2, to=2-2]
	\arrow["\framing", from=2-1, to=2-2]
\end{tikzcd}\]
where the bottom horizontal map is the morphism induced by $\framing$. 
\end{definition}
\noindent We define the stack $\pairs$ of $\beta$--stable $\framing$--pairs defined via the pullback
  \begin{equation}\label{eqn : def of stack of stable pairs}\begin{tikzcd}
	\pairs & \flags \\
	{\abetatorstack \times \abetatfstack} & {\perfstack \times \perfstack}
	\arrow[from=1-1, to=1-2]
	\arrow[from=1-1, to=2-1]
	\arrow["{\partial_1[1]\times\partial_0}", from=1-2, to=2-2]
	\arrow[from=2-1, to=2-2]
\end{tikzcd}\end{equation}
  where $\partial_1[1]$ denotes the composition of the morphism induced by $\partial_1$ and the shift morphism. Compatibly with the notation \eqref{eqn : delta maps}, for any object $\operatorname{P}= [\framing \rightarrow F \rightarrow T]$ in $\pairs$ we denote
  \[
        \partial_1(\operatorname{P}) = F \qquad \partial_0(\operatorname{P}) = T\ .
  \]

  \begin{proposition}
      The stack $\pairs$ is a quasi--separated, geometric derived stack, locally almost of finite presentation over ${Spec}(\BC)$.
  \end{proposition}

  \begin{remark}\label{rmk : topological type of pairs}
      The stacks $\flags$ and $\pairs$ have a natural decomposition in open and closed substacks, coming from the toplogical type:
      \[
        \flags = \bigsqcup_{r,d}\boldsymbol{\operatorname{Perf}}^{\dagger}_{r,d}(X,\framing) \qquad \pairs = \bigsqcup_{r,d}\boldsymbol{\operatorname{P}}^{\beta}_{r,d}(X,\framing) \ .
      \]
  \end{remark}

\subsection{Quot--spaces in $\abeta$ }\label{subsec : Quotspaces}
\noindent In this Subsection, we assume that the framing object $\framing$ belongs to the heart $\abeta$.

\begin{definition}\label{def : der quor}
    Let $S$ be an affine scheme over $\BC$, and let $\framing$ be a $\tau^\beta$--flat compactly supported complex of coherent sheaves on $X$. We define the \textit{derived Quot--space} associated to $\framing$ as the derived pullback
    \begin{equation}\label{eqn : der Quot}
        \begin{tikzcd}
	\dQuot & {\segal_2\cohflat(X,\tau^\beta)} \\
	S & {\cohflat(X,\tau^\beta)}
	\arrow[from=1-1, to=1-2]
	\arrow[from=1-1, to=2-1]
	\arrow["\lrcorner"{anchor=center, pos=0.125}, draw=none, from=1-1, to=2-2]
	\arrow["{\partial_1}", from=1-2, to=2-2]
	\arrow[from=2-1, to=2-2, "\framing"]
\end{tikzcd}
    \end{equation}
\end{definition}

\noindent Unravelling the definition, we see that for any $S\in \operatorname{dAff}_{\BC}$, the Quot space $\dQuot$ parametrizes fiber sequences
\begin{equation}\label{eqn : fiber sequence in derived quot}
    K \rightarrow \framing \rightarrow E
\end{equation}
where $K$ and $E$ are flat with respect to the t--structure $\tau$. In particular, when $S$ is underived, the objects $K,\framing,E$ belong to the heart $\abeta$, so \eqref{eqn : fiber sequence in derived quot} becomes a short exact sequence in $\abeta$. We deduce that $\dQuot$ is the standard derived enhancement of the Quot space of $\framing$ in the heart $\abeta$ defined in \cite{Bayer_2021}. In particular, when $S$ is an underived scheme, we get an algebraic space $\operatorname{Quot}^\beta_S(X,\framing)$ locally of finite presentation over $S$. \newline
\noindent Thanks to Lemma \ref{lemma : notherianity} and Remark 2.14 in \cite{Rota_2021} we also have the following.  

\begin{proposition}
    Let $\beta\in \BQ$. Then, the Quot--space $\dQuot$ satisfies the strong existence part and the uniqueness part of the valuative criterion of properness, in the sense of Definition 11.8 in \cite{Bayer_2021}.
\end{proposition}

\noindent We end the section with a piece of notation. We denote as
\[
    \Quotbeta = \bigsqcup_{r,d} \Quotbetard
\]
the decomposition into connected components given by the Chern class of $E$ in the Quot space $\Quotbeta$.

\subsection{Stable pairs Vs Quot spaces Vs Bradlow pairs}\label{subsec : comparing with Quot}
In this section we compare the stack $\pairs$ of $\beta$--stable $\framing$--pairs with other moduli spaces occurring in the literature.
\bigskip

\noindent\textbf{Comparing stable pairs and Quot spaces.} In some cases, $\beta$--stable $\framing$--pairs can be identified with quotients in the heart $\abeta$.

    \begin{proposition}\label{prop : PairsQuot}
    Let $\linebundle$ be a line bundle on $X$ of degree smaller than $\beta$, and consider a fiber sequence
    \[
      \operatorname{P}=[\linebundle \xrightarrow{s} F\rightarrow T]  
    \]
    in $\perf$. Then the following are equivalent
    \begin{enumerate}
        \item\label{item : stpair} $\operatorname{P}$ defines a $\beta$--stable $\linebundle$--pair.
        \item\label{item : quotient} $s[1]$ is surjective in $\abeta$.
    \end{enumerate}
    
    \end{proposition}
    \begin{proof} We first prove that \ref{item : stpair} implies \ref{item : quotient}. Let $\linebundle \xrightarrow{s} F\rightarrow T$ be a $\beta$--stable $\linebundle$--pair, where $\linebundle$ is a line bundle of slope smaller than $\beta$. Then the rotated sequence
    \[
        T \rightarrow \linebundle[1] \xrightarrow{s[1]} F[1]
    \]
    is a fiber sequence whose terms are in the heart $\abeta$. In particular it is a short exact sequence in $\abeta$, and $s[1]$ is surjective in $\abeta$.

    \medskip
    
    \noindent We now prove that \ref{item : quotient} implies \ref{item : stpair}. Suppose that $\linebundle[1] \xrightarrow{s[1]}F[1]$ is a quotient in $\abeta$ and assume that $\linebundle$ is a line bundle of slope smaller than $\beta$. Then, $\linebundle[1] \in \abetator$, and Lemma \ref{lemma : torsionproperties} implies that $F[1]\in \abetator$. Now consider the short exact sequence
    \begin{equation}
        0\rightarrow K \rightarrow \linebundle[1] \xrightarrow{s[1]}F[1]\rightarrow 0
    \end{equation}
    in $\abeta$. Taking the long exact sequence in cohomology with respect to the standart t--structure yields
    \begin{equation}\label{eqn : longexact}
        0\rightarrow\mathcal{H}^{-1}(K) \rightarrow \mathcal{H}^{-1}(\linebundle[1]) \rightarrow \mathcal{H}^{-1}(F[1]) \rightarrow \mathcal{H}^0(K) \rightarrow0
    \end{equation}
    In the exact sequence of sheaves \eqref{eqn : longexact}, we have natural identifications $\mathcal{H}^{-1}(\linebundle[1]) \cong \linebundle \in \fbeta$, and $\mathcal{H}^{-1}(F[1]) \cong F \in \fbeta$. Since $F$ is torsion--free and $s[1]$ is nonzero, the injection $0\rightarrow\mathcal{H}^{-1}(K) \rightarrow \linebundle$ implies that $\mathcal{H}^{-1}(K)=0$. Thus, $K \cong \mathcal{H}^0(K)$ is identified with the cokernel of $s$ as a morphism of coherent sheaves.
    \end{proof}
\noindent In particular the classical truncation of the moduli stack of $\beta$--stable $\linebundle$--pairs is the Quot--space of $\linebundle$
\begin{equation}\label{eqn : QuotPairs}
        \Quotbetaone \cong \prescript{cl}{}\pairslb\ .
\end{equation}

\bigskip

\noindent\textbf{Comparing Stable Pairs and Bradlow pairs. }We fix a coherent sheaf $\framing \in \coh(X)$ whose Chern class is $(r,d)$ and we consider morphisms of the form $\framing\rightarrow F$.

\begin{definition}\label{def : bradlow pair}
    Let $\sigma \in \BR$. A pair
    \[
        (F, s\colon \framing \rightarrow F)
    \]
    is a $\sigma$--(semi)stable Bradlow pair of rank $r$ and degree $d$ if $F \in \coh(X)$ has Chern class $(r,d)$ and and the following holds.
    \begin{enumerate}
        \item $\mu(F') < (\leq) \mu(F)+ \frac{\sigma}{\operatorname{rank}(F)}$ for any subsheaf $F'\subset F$.
        \item $\mu(F') + \frac{\sigma}{\operatorname{rank}(F')}< (\leq) \mu(F) + \frac{\sigma}{\operatorname{rank}(F)}$ for any subsheaf $F$ of $F$ such that $F'\subset \operatorname{Im}(s)$.
    \end{enumerate}
\end{definition}

\noindent This definition was first given by Bradlow \cite{bradlow1993birationalequivalencesvortexmoduli} in the case where $\framing = \mathcal{O}_X$. Successively, Lin \cite{Lin_2018} extended the definition to a more general setting. In particular, it is shown that the stack parametrizing semistable Bradlow pairs in the sense of Definition \ref{def : bradlow pair} of rank $r$ and degree $d$ admits a projective \textit{coarse} moduli space 
\begin{equation}\label{eqn : coarse bradlow}
   \bradstack \longrightarrow \brad
\end{equation}
Moreover, the moduli functor of stable Bradlow pairs is represented by an open subscheme $\operatorname{B}^{\sigma -s}_{r,d}(X,\framing) \subset \brad$. The relation of the aforementioned moduli space with Quot spaces in $\abeta$ has been explored in \cite{10.4134/JKMS.J230331} and \cite{Rota_2021}, following the ideas  in \cite{bridgeland2020hallalgebrascurvecountinginvariants}. In particular, the following Lemma follows immediatly from Lemma 4.3 in \cite{Rota_2021} --- See also Lemma 2.3 in \cite{bridgeland2020hallalgebrascurvecountinginvariants} and Lemma 2.2 in \cite{10.4134/JKMS.J230331}. Fix and a real parameter $\beta \geq \frac{d}{r}$, and recall the function $\mu_-$ defined in \eqref{eqn : mu-}. 

\begin{lemma}\label{lem : bradlow vs quot}
    Suppose that $F \in \coh(X)$ has rank $r$ and degree $d$, and let $\sigma = \beta r - d$, and assume that $\framing \in \fbeta$. Then an epimorphism $s[1] \colon \framing[1] \rightarrow F[1]$ in $\abeta$ defines a semistable Bradlow pair $s \colon \framing \rightarrow F$. Viceversa, a semistable Bradlow pair $s \colon \framing \rightarrow F$ defines a quotient $s[1] \colon \framing[1] \rightarrow F[1]$ in $\abeta$ whenever $\mu_-(\operatorname{coker}(s)) < \beta$. 
\end{lemma}

\noindent By using this result, it is shown in \cite{10.4134/JKMS.J230331} that if $\beta$ is chosen in such a way that there is no strictly $\sigma$--semistable Bradlow pair, then we have an isomorphism
\[
    \operatorname{B}^{\sigma -s}_{r,d}(X,\framing)  \cong {\operatorname{Quot}_{r,d}^{\beta}(X,\mathcal{V}[1])} \  
\]
This morphism sends a $\sigma$--semistable Bradlow pair $(F,s)$ to the map $s[1]$. We compare the coarse moduli space $\brad$ with the Quot space ${\operatorname{Quot}_{r,d}^{\beta}(X,\mathcal{V}[1])}$ when $\framing = \linebundle$ is a line bundle of slope smaller then $\beta$. 

\begin{proposition}\label{prop : BradlowQuot}
    Fix a real number $\beta$, and let $\sigma = \beta r - d$. Let $\linebundle$ be a line bundle of slope smaller that $\beta$. Then there is an open embedding of schemes
    \[
        \Quotbetaonerd \hookrightarrow \bradlb 
    \]
    for any Chern class $(r,d)$.
\end{proposition}
\begin{proof}
    Thanks to Lemma \ref{lem : bradlow vs quot} we have an embedding of the Quot space $\Quotbetaonerd$ inside the stack $\bradstacklb$. Since the Quot space is representable, we have an induced embedding in the coarse moduli space $\bradlb$. Moreover, we know that such embedding is cut out by the condition
    \begin{equation}\label{eqn : condition bradlow vs quot}
        \mu_-(\operatorname{coker}(s)) < \beta
    \end{equation}
   which we now prove to be open. Let $\widetilde{\tau}^\beta$ be the tilting of the standard t--structure by the torsion pair \eqref{eqn : tweaked tp}. Then, we can define a stack $\cohflat_{\operatorname{t.f.}}(X,\widetilde{\tau}^{\beta})$ parametrizing objects in $\widetilde{\tbeta}$. The inclusion $\tbeta \subseteq \widetilde{\tbeta}$ induces an embedding
    \begin{equation}\label{eqn : inclusion of tbetas}
        \cohflat_{\operatorname{t.f.}}(X,{\tau}^{\beta}) \hookrightarrow \cohflat_{\operatorname{t.f.}}(X,\widetilde{{\tau}}^{\beta})
    \end{equation}

    \medskip
    
    \noindent Observe that if $\linebundle$ is a line bundle, then any nonzero morphism $s \colon \linebundle \rightarrow F$ is injective. Thus, we have a well defined map
    \begin{equation}\label{eqn : coker map for bradlow pairs}
        \operatorname{coker} \colon \bradlb \longrightarrow \cohflat_{\operatorname{t.f.}}(X,\widetilde{\tau}^{\beta})
    \end{equation}
     which can be described on closed points by sending a semistable pair $(F,s)$ to $\operatorname{coker}(s)$. Moreover, we know from the proof of Proposition \eqref{prop : PairsQuot} that quotients of $\linebundle[1]$ in $\abeta$ are of the form
    \[
        0\rightarrow T \rightarrow \linebundle [1] \rightarrow F[1] \rightarrow 0
    \]
    where $F \in \fbeta$ and $T\in \tbeta$. Thus we have a well defined map
      \begin{equation}\label{eqn : ker map for quot}
        \operatorname{ker} \colon \Quotbetaonerd \longrightarrow \cohflat_{\operatorname{t.f.}}(X,\widetilde{\tau}^{\beta})
    \end{equation}
    which projects onto the kernel $T$. Then, condition \eqref{eqn : condition bradlow vs quot} can be restated by saying that the square
\[\begin{tikzcd}
	\Quotbetaonerd & {\cohflat_{\operatorname{tor}}(X,{\tau}^{\beta})} \\
	\bradlb & {\cohflat_{\operatorname{tor}}(X,\widetilde{\tau}^{\beta})}
	\arrow["{\operatorname{ker}}"',from=1-1, to=1-2]
	\arrow[hook, from=1-1, to=2-1]
	\arrow[hook, from=1-2, to=2-2]
	\arrow["{\operatorname{coker}}"', from=2-1, to=2-2]
\end{tikzcd}\]
is a pullback. Finally, arguing as in the proof of Proposition \ref{prop : open tp}, we deduce that the right vertical map is cut out by the condition 
    \[
        \mu_- > \beta
    \]
    The lower semicontinuity of the function $\mu_-$ implies that the inclusion \eqref{eqn : inclusion of tbetas} is open. 
\end{proof}
\noindent Observe that the above Theorem holds for $\beta \in \BR$. Indeed, we know that the Quot space $\Quotbetaonerd$ is representable. However, if $\beta \notin \BQ$, we can't conclude that it is proper.

\section{Left action}
Let $X$ be a curve as before and let $\beta \in \BQ$. We define a left action of the categorical Hall algebra of $\cohbetass$ on the pro--category $\pro(\pairs)$ defined in Subsection \ref{subsec : derived stack of pairs}. Throughout this section, we fix an object $\framing$ in $\perf$.

\subsection{Construction of the left action}\label{subsec : Construction of the left action}

\noindent\textbf{Left Hecke pattern. }We define the stack which induces the convolution diagram \eqref{action} from the introduction in the context of our left action. We denote as $\cohbetass[1]$ the shift by $1$ of the category $\cohbetass$.

\begin{definition}[Extension of stable pairs]\label{def : left extension}
    Let $G\in \cohbetass[1]$ and let $\operatorname{P} = [\mathcal{V}\rightarrow F \rightarrow T]$ be a $\beta$--stable pair. An extension of $\operatorname{P}$ by $G$ is a diagram in $\perf$ of the form
\begin{equation}\label{eqn : left extension}
    \begin{tikzcd}
0 \arrow[r]    & \mathcal{V}  \arrow[d] \arrow[r] & F \arrow[d] \arrow[r] & F' \arrow[d] \\ & 0 \arrow[r] & T \arrow[d] \arrow[r] & T' \arrow[d]\\ & & 0 \arrow[r] & G \arrow[d]\\ &&& 0
\end{tikzcd}
\end{equation}
all of whose squares are pullbacks. Moreover we require that $\operatorname{P}'=[\mathcal{V}\rightarrow F' \rightarrow T']$ is a $\beta$--stable pair. For an extension ${\operatorname{E}}$ of the form \eqref{eqn : left extension}, we denote
\begin{equation}\label{eqn : extensionprojection}
\omega_1({\operatorname{E}}) = \operatorname{P} \hspace{1cm} \omega_0({\operatorname{E}}) = \operatorname{P}' \hspace{1cm} u^l_1({\operatorname{E}}) = G\ .  
\end{equation}
\end{definition}

\noindent The following Proposition is a formal consequence of the above definitions and the properties of torsion pairs. It is the key property underlying the associativity of left action in Theorem \ref{theo 2}. Compare with III.4.19 in \cite{dps}. 

\begin{proposition}[Left Hecke Pattern]\label{prop : leftthecke}
     Consider a diagram of the form \eqref{eqn : left extension}. Let $G\in \cohbetass[1]$ and let $\operatorname{P}' = [\mathcal{V}\rightarrow F' \rightarrow T']$ be a $\beta$--stable pair. Then $\operatorname{P} = [\mathcal{V}\rightarrow F \rightarrow T]$ is a $\beta$--stable pair if and only if $T\in \abeta$.
\end{proposition}
\begin{proof}
    Observe that $\cohbetass[1] \subset \abetator$. If the pair $\operatorname{P}$ is stable then $T\in \abeta$ by Definition, hence we prove the other implication.\newline
    \noindent Suppose that $\operatorname{P}'$ is a $\beta$--stable $\framing$--pair and $T\in \abeta$. In particular the fiber sequence $T\rightarrow T'\rightarrow G$ is a short exact sequence in $\abeta$. Then, $T'\in \abetatf$ implies that $T\in \abetatf$ thanks to Lemma \ref{lemma : torsionproperties}. Now consider the fiber sequence $F\rightarrow F'\rightarrow G$. Rotating we get the fiber sequence
    \begin{equation}\label{eqn : shiftedsec}
        G \rightarrow F[1]\rightarrow F'[1]        
    \end{equation}
    where $G\in \abetator$ and $F'[1]\in \abetator$. Thus, the sequence \eqref{eqn : shiftedsec} is a short exact sequence in $\abeta$. In particular Lemma \ref{lemma : torsionproperties} implies that $F[1]\in \abetator$.
\end{proof}

\bigskip

\noindent\textbf{Summary. }We now sketch the construction of the left action. We will define in the following Subsections a stack $\segal^l_1\pairs$ parametrizing extensions of stable pairs in the sense of Definition \ref{def : left extension}. As before, we denote as $\cohbetassstack[1]$ the image of the stack $\cohbetassstack$ under the shift morphism $[1] \colon \perfstack \rightarrow \perfstack$. Then we have a convolution diagram
\begin{equation}\label{eqn : left action introduction}
\begin{tikzcd}
	& {\segal_1^l\pairs} &&& {\operatorname{E}} \\
	{\cohbetassstack[1] \times \pairs} && \pairs & {(G,\operatorname{P})} && {\operatorname{P}'}
	\arrow["{u^l_1 \times \omega_1}", from=1-2, to=2-1]
	\arrow["{\omega_0}"', from=1-2, to=2-3]
	\arrow[maps to, from=1-5, to=2-4]
	\arrow[maps to, from=1-5, to=2-6]
\end{tikzcd}
\end{equation}
We have a natural isomorphism $\cohbetassstack \cong \cohbetassstack[1]$ induced by the shift operation. We prove further that the map $u^l_1 \times \omega_1$ is lci, and the map $\omega_0$ is proper. 
 \begin{remark}\label{rmk : coh ss and topology}
      An important ingredient in the proofs contained in this Section is the following observation. If $G\in \cohbetass$, then $G$ belongs both to $\fbeta$, and $\mathcal{T}^{\beta-\epsilon}$ for any $\epsilon >0$. This observation underlies many of the geometric properties of the morphisms \eqref{eqn : left action introduction}. For example see the proof of Lemma \ref{lemma : aux rpas}. 
 \end{remark}
\noindent The pull--push operation
\[
    \omega_{0*}(\omega_1 \times u^l_1)^* \colon H^{BM}_*(\cohbetassstack)\otimes H^{BM}_*(\pairs) \longrightarrow H^{BM}(\pairs)
\]
induces the structure of a left module for the \textit{CoHA} attached to the category $\cohbetass$ on the space $H^{BM}_*(\pairs)$. In particular, the associativity property of this action is encoded in the statement of Proposition \ref{prop : leftthecke}. Even more, we provide a categorification of such action. In order to do this, we upgrade the diagram \eqref{eqn : left action introduction} to a more complicated object, as we explain in the next Subsection.

  \subsection{Stack of left extensions}

   Recall the stack of $\framing$--pairs $\flags$ of Definition \ref{def : perf}. Construction I.3.6 in \cite{dps} provides us with a simplicial derived stack
\begin{equation}\label{eqn : flags relative}
    \waldhausen^l\flags 
\end{equation}
Unravelling the definition, the stack $\segal^l_1\flags$ fits into the fiber product
\begin{equation}\label{eqn : fiber product defining S1flags}
\begin{tikzcd}
	{\segal^l_1\flags} & {\segal_3\perfstack} \\
	{Spec(\BC)} & {\perfstack}
	\arrow[from=1-1, to=1-2]
	\arrow[from=1-1, to=2-1]
	\arrow["{\partial_{23}}", from=1-2, to=2-2]
	\arrow["\framing", from=2-1, to=2-2]
\end{tikzcd}
\end{equation}
where the bottom horizontal map is the morphism induced by the object $\framing$, and the map $\partial_{23} \colon \segal_3\perfstack \rightarrow \segal_1\perfstack\cong \perfstack$ is induced by the face map $[3] \rightarrow[1]$ in $\Delta^{op}$ which avoids $2$ and $3$. In particular, the stack $\segal^l_1\flags$ parametrizes extensions of the form \eqref{eqn : left extension}, where no condition is imposed on the perfect complexes occurring in the diagram. Moreover, the fiber diagram \eqref{eqn : fiber product defining S1flags} induces natural maps
\begin{equation}\label{eqn : convolution left}
\begin{tikzcd}
	& {\segal_1^l\flags} &&& {\operatorname{E}} \\
	{\perfstack \times \flags} && \flags & {(G,\operatorname{P})} && {\operatorname{P}'}
	\arrow["{u^l_1 \times \omega_1}", from=1-2, to=2-1]
	\arrow["{\omega_0}"', from=1-2, to=2-3]
	\arrow[maps to, from=1-5, to=2-4]
	\arrow[maps to, from=1-5, to=2-6]
\end{tikzcd}
\end{equation}
 We now discuss the higher simplicial levels. The morphism $u_1^l$ can be extended to a morphism of simplicial stacks
 \begin{equation}\label{eqn : relative simp flags}
     \waldhausen^l\flags \longrightarrow \waldhausen\perfstack
 \end{equation}
 which satisfies the \textit{relative 2--Segal property}. As the 2--Segal property encodes the Hall algebra structure, the relative 2--Segal property encodes the Hall algebra action --- See Corollary 5.4 in \cite{godicke2024infty} for a precise formulation. 
 \bigskip

\noindent Proceeding as in II.3.2 in \cite{dps} --- see  Construction I.5.5 --- we can define a simplicial stack $\waldhausen^l\pairs$, which parametrizes extensions as in Definition \ref{def : left extension}. Moreover we have a natural morphism
\begin{equation}\label{eqn : relative simp pairs}
    \waldhausen^l\pairs \longrightarrow \waldhausen \cohbetassstack[1]
\end{equation}
In particular, the maps \eqref{eqn : left action introduction} are well--defined. The crucial step in the proof of Theorem \ref{theo 2} is that the morphism \eqref{eqn : relative simp pairs} also satisfies the relative 2--Segal property.

\begin{proposition}\label{prop : left segal}
    The morphism of derived stacks \eqref{eqn : relative simp pairs} is a relative 2--Segal space.
\end{proposition}

\noindent The 2--Segal property of the morphism \eqref{eqn : relative simp pairs} is induced by \eqref{eqn : relative simp flags} in two steps. We follow the procedure in the proof of Proposition III.4.21 in \cite{dps}.
\medskip

\noindent \textbf{Step 1. }Consider the auxiliary stack $\boldsymbol{\operatorname{Perf}}^{\dagger}_{0-\abeta}(X,\framing)$ defined via the pullback
    \begin{equation}\label{eqn: aux stack}
    \begin{tikzcd}
	{\boldsymbol{\operatorname{Perf}}^{\dagger}_{0-\abeta}(X,\framing)} & \flags \\
	{\cohflat(X,\tau^\beta)} & \perfstack
	\arrow[from=1-1, to=1-2]
	\arrow[ from=1-1, to=2-1]
	\arrow["{\partial_0}", from=1-2, to=2-2]
	\arrow[""{name=0, anchor=center, inner sep=0}, from=2-1, to=2-2]
	\arrow["\lrcorner"{anchor=center, pos=0.125}, draw=none, from=1-1, to=0]
\end{tikzcd}
    \end{equation}
    where the t--structure $\tau^\beta$ was defined in Subsection \ref{subsec : tilted heart abeta}. Such stack parametrizes fiber sequences $[\framing\rightarrow F \rightarrow T]$ in $\perf$ with the only condition of $T$ being in the heart $\abeta$. Applying construction II.3.2 in \cite{dps} as before, we define the relative simplicial stack
    \begin{equation}\label{eqn : aux relative stack}
        u_{\bullet}^l \ \colon \waldhausen^l\boldsymbol{\operatorname{Perf}}^{\dagger}_{0-\abeta}(X,\framing) \longrightarrow \waldhausen\cohbetassstack[1] \ .
    \end{equation}

\begin{lemma}\label{lemma : step 1}
The square
\begin{equation}\label{eqn : aux square}
\begin{tikzcd}
{\segal_1^l\boldsymbol{\operatorname{Perf}}^{\dagger}_{0-\abeta}(X,\framing)} & {\segal^l_1\flags} \\
	{\cohbetassstack[1] \times \boldsymbol{\operatorname{Perf}}^{\dagger}_{0-\abeta}(X,\framing)} & {\perfstack \times \flags}
	\arrow[from=1-1, to=1-2]
	\arrow["{u_1^l \times \omega_1}"', from=1-1, to=2-1]
	\arrow["{u_1^l \times \omega_1}", from=1-2, to=2-2]
	\arrow[from=2-1, to=2-2]
\end{tikzcd}
\end{equation}
is a pullback. Moreover, the morphism \eqref{eqn : aux relative stack} is a relative 2--Segal space.
\end{lemma}
\begin{proof}
    The second statement follows from the first one and Proposition II.3.6 in \cite{dps}. Lemma \ref{lemma : notherianity} tells us that the horizontal maps in \eqref{eqn : aux square} are representable by Zariski--open embeddings. Thus we can prove the first statement by evaluating at geometric points. Unravelling the definition, we have to show that if we are given a diagram of the form \eqref{eqn : left extension} where $T \in \abeta$ and $G \in \cohbetass[1]$, then $T' \in \abeta$. This follows directly from the fact that $T \rightarrow T'\rightarrow G$ is a fiber sequence.
\end{proof}

\noindent\textbf{Step 2. }The stack $\pairs$ maps naturally to the stack $\boldsymbol{\operatorname{Perf}}^{\dagger}_{0-\abeta}(X,\framing)$. Moreover, we have induced maps at the higher simplicial levels. Thus, we can set up a square
\begin{equation}\label{eqn : non aux square}
\begin{tikzcd}
{\segal_1^l\pairs} & {\segal_1^l\boldsymbol{\operatorname{Perf}}^{\dagger}_{0-\abeta}(X,\framing)} \\
	{\pairs} & {\boldsymbol{\operatorname{Perf}}^{\dagger}_{0-\abeta}(X,\framing)}
	\arrow[from=1-1, to=1-2]
	\arrow["{\omega_0}"', from=1-1, to=2-1]
	\arrow["{\omega_0}", from=1-2, to=2-2]
	\arrow[from=2-1, to=2-2]
\end{tikzcd}
\end{equation}

\begin{lemma}\label{lemma : step 2}
    The square \eqref{eqn : non aux square} is a pullback.
\end{lemma}
\begin{proof}
    The horizontal arrows are representable by open embeddings. Thus we prove the statement by evaluating at closed points. Unravelling the definitions, we have to prove that if we are given a diagram of the form \eqref{eqn : left extension}, where $G \in \cohbetass[1],  \ [\framing \rightarrow F' \rightarrow T']$, is a $\beta$--stable pair, and $T\in \abeta$, then  $[\framing \rightarrow F \rightarrow T]$, is a $\beta$--stable pair. This is exactly the content of Proposition \ref{prop : leftthecke}.
\end{proof}
\noindent It follows from Lemma \ref{lemma : step 2} that the square 
\begin{equation*}
\begin{tikzcd}
{\segal_1^l\pairs} & {\segal_1^l\boldsymbol{\operatorname{Perf}}^{\dagger}_{0-\abeta}(X,\framing)} \\
	{\cohbetassstack[1] \times \pairs} & {\cohbetassstack[1] \times \boldsymbol{\operatorname{Perf}}^{\dagger}_{0-\abeta}(X,\framing)}
	\arrow[from=1-1, to=1-2]
	\arrow["{u_1^l \times \omega_1}"', from=1-1, to=2-1]
	\arrow["{u_1^l \times \omega_1}", from=1-2, to=2-2]
	\arrow[from=2-1, to=2-2]
\end{tikzcd}
\end{equation*}
is a pullback. Then, Lemma \ref{lemma : step 1} allows us to apply Corollary I.5.7 in \cite{dps} to conclude the proof of Proposition \ref{prop : left segal}.

\subsection{Proof of the left action}
In this subsection we complete the proof of the left action part in Theorem \ref{theo 2} by checking the geometric properties.
\medskip 

\noindent \textbf{Proper push--forward. }Proceeding as in Notation III.4.22 in \cite{dps}, we can define a derived stack $$\segal_1^l\boldsymbol{\operatorname{FlagCoh}}^{(1)}_{\operatorname{t.f.,\beta-ss[1]}}(X,\tau^\beta)$$
    whose closed points are short exact sequences $0 \rightarrow T\rightarrow T'\rightarrow G \rightarrow 0$ in $\abeta$ where $T, T'\in \abetatf$ and $G\in \cohbetass[1]$. Consistently with the above notations, we set
    \[
        \partial_0(T\rightarrow T'\rightarrow G)=G,\ \  \partial_1(T\rightarrow T'\rightarrow G)=T', \ \ \partial_2(T\rightarrow T'\rightarrow G)=T \ .
    \]

    \begin{lemma}\label{lemma : aux rpas}
        The morphism 
        \[
            \partial_1 \colon \segal_1^l\boldsymbol{\operatorname{FlagCoh}}^{(1)}_{\operatorname{t.f.,\beta-ss[1]}}(X,\tau^\beta) \rightarrow \abetatfstack
        \]
        is locally rpas.
    \end{lemma}
    \begin{proof}
        Let $T'\in \abetatf$. The fiber $\partial_1^{-1}(T')$ is a stack parametrizing short exact sequences $0 \rightarrow T\rightarrow T'\rightarrow G \rightarrow 0$ in $\abeta$ such that $T\in \abetatf$ and $G\in \abetator$. Thanks to Lemma \ref{lemma : torsionproperties}, every such short exact sequence has the property $T\in \abetatf$. Thus we get an embedding
        \begin{equation}\label{eqn : fiber in Quot}
            \partial_1^{-1}(T') \hookrightarrow\operatorname{Quot}^\beta(X,T') \ .
        \end{equation}
       In particular, the morphism $\partial_1$ is representable by algebraic spaces, and we have to show that the embedding \eqref{eqn : fiber in Quot} is closed. We have a decomposition 
        \[
            \operatorname{Quot}^\beta(X,T') = \bigsqcup_{r,d}\operatorname{Quot}_{r,d}^\beta(X,T')
        \]
        according to the topological type of the quotient $G$ of $T'$. The image of the embedding \eqref{eqn : fiber in Quot} is contained in those connected components $\operatorname{Quot}_{r,d}^\beta(X,T')$ such that $d/r = \beta$. Consider a closed point
        $0\rightarrow T\rightarrow T'\rightarrow G \rightarrow 0$ in one of these connected components, which we denote as $\operatorname{Quot}_{r_0,d_0}^\beta(X,T')$. Then, the complex $G$ fits in a canonical short exact sequence in $\abeta$ 
        \[
            0\rightarrow \overbrace{\mathcal{H}^{-1}(G)}^{\in \fbeta}[1]\rightarrow G \rightarrow \overbrace{\mathcal{H}^{0}(G)}^{\in\tbeta} \rightarrow 0
        \]
        There are two mutually exclusive cases.
        \begin{enumerate}
            \item $G \cong \mathcal{H}^{-1}(G)[1] \in \cohbetass[1]$.
            \item\label{item : complement} $\mu(\mathcal{H}^{-1}(G)) < \beta$.
        \end{enumerate}
        In particular, we see that the complement of the fiber $\partial_1^{-1}(T')$ in $\operatorname{Quot}_{r_0,d_0}^\beta(X,T')$ is cut out by condition \ref{item : complement}. According to Lemma \ref{lemma : cohbetacirc}, the latter condition is open. It follows that the complement is closed, hence \eqref{eqn : fiber in Quot} is cloded. Finally, the above proof extends to general $S$--fibers, for $S$ an arbitrary algebraic space, given the analysis of Subection \ref{subsec : Quotspaces}. 
    \end{proof}

\noindent Consider the map 
    \begin{equation}\label{eqn : pitfzero}
        \pitfzero \colon= \partial_0 \ \colon \ \pairs \longrightarrow \abetatfstack
    \end{equation}
    sending a stable pair $[\framing \rightarrow F \rightarrow T]$ to the $\abeta$--torsion free component $T$. We also have a map
    \[
                \pi_1^{\operatorname{t.f.}} \ \colon \segal_1^l\pairs \longrightarrow \segal_1^l{\boldsymbol{\operatorname{FlagCoh}}}^{(1)}_{\operatorname{t.f.,\beta-ss[1]}}(X,\tau^\beta)\
    \]
    which sends an extension \eqref{eqn : left extension} to the sub--extension $0 \rightarrow T\rightarrow T'\rightarrow G \rightarrow 0$.

    \begin{proposition}
        The square
        \begin{equation}\label{eqn : left properness square}
        \begin{tikzcd}
	{\segal_1^l \pairs} & {\segal_1^l}{\boldsymbol{\operatorname{FlagCoh}}}^{(1)}_{\operatorname{t.f.,\beta-ss[1]}}(X,\tau^\beta) \\
	\pairs & \abetatfstack
	\arrow[from=1-1, to=1-2 , "\pi_1^{\operatorname{t.f.}}"]
	\arrow[from=1-1, to=2-1, "\omega_0"]
	\arrow[from=1-2, to=2-2, "\partial_0"]
	\arrow[""{name=0, anchor=center, inner sep=0}, from=2-1, to=2-2, "\pitfzero"]
    \end{tikzcd}
    \end{equation}
    is a pullback. In particular, the left vertical map is locally rpas.
    \end{proposition}
    \begin{proof}
        The proof of the first statement closely follows the proof of Lemma III.4.23 in \cite{dps}, where their Proposition III.4.19 is replaced by our Proposition \ref{prop : leftthecke}. The second statement follows from the first one and Lemma \ref{lemma : aux rpas}.
     \end{proof}

\bigskip

\noindent\textbf{Flat pull--back. }We consider the map
\begin{equation}\label{eqn : left pullback map}
    u_1^l \times \omega_1 \colon  \segal^l_1 \pairs \longrightarrow \cohbetassstack [1] \times \pairs
\end{equation}

\begin{proposition}
    Let $x \colon Spec(\BC) \rightarrow \segal_1^l\pairs$ be a point classifying an extension $\operatorname{E}$ of the form \eqref{eqn : left extension}. Let $\mathbb{T}_x$ be the tangent complex at $x$ of the morphism \eqref{eqn : left pullback map}. Then
    \begin{equation}\label{eqn : h0 is 0}
        \operatorname{H}^2(\mathbb{T}_x)  =0 
    \end{equation}
    In particular, the map \eqref{eqn : left pullback map} is derived lci. Moreover the map \eqref{eqn : left pullback map} is quasi--compact and finitely connected.
\end{proposition}
\begin{proof}
    As a consequence of Corollary II.3.26 in \cite{dps}, we have a canonical fiber sequence
    \[
        \mathbb{T}_x \rightarrow \mathbb{R}\operatorname{Hom}_X(F',F)[1] \oplus \mathbb{R}\operatorname{Hom}_X(G,F')[1] \rightarrow \mathbb{R}\operatorname{Hom}_X(F',F')[1]  \   .
    \]
Passing to the associated long exact sequence in cohomology yields
\begin{equation}\label{eqn : long exact sequence lci}
   0 \rightarrow \operatorname{H}^{-1}(\mathbb{T}_x) \rightarrow \dots \rightarrow \operatorname{Ext}^2(F',F') \rightarrow \operatorname{H}^2(\mathbb{T}_x) \rightarrow\operatorname{Ext}^3(F',F) \oplus\operatorname{Ext}^3(G,F') \rightarrow 0
\end{equation}
Since $F '$ and $F$ are in $\fbeta$, we have 
$\operatorname{Ext}^2(F',F')  =0$ and $\operatorname{Ext}^3(F',F) =0$. Moreover, since $G \in \cohbetass[1]$ we have
\[
    \operatorname{Ext}^3(G,F') \cong \operatorname{Ext}^2(G[-1],F') =0 \ .
\]
    It follows from the long exact sequence \eqref{eqn : long exact sequence lci} that $\operatorname{H}^2(\mathbb{T}_x) =0$. We conclude that the complex $\mathbb{T}_x$ has Tor--amplitude $[-1,1]$, and the map \eqref{eqn : left pullback map} is derived lci. The last statement follows from Corollary II.3.24 in \cite{dps}.
\end{proof}

\begin{remark}
    If $\framing = \linebundle[1]$, where $\linebundle$ is a line bundle of degree smaller then beta, we can give an alternative proof of the above Proposition. Indeed, the square 
    \begin{equation*}\label{eqn : left flatness square}
        \begin{tikzcd}
	{\segal_1^l \pairs} & {\segal_1^l\boldsymbol{\operatorname{FlagCoh}}^{\operatorname{(1)}}_{\operatorname{tor, \beta-ss[1]}}}(X,\tau^\beta) \\
	\pairs \times \cohbetassstack[1] & \abetatorstack \times \cohbetassstack[1]
	\arrow[from=1-1, to=1-2 , "\pi_1^{\operatorname{tor}[-1]}"]
	\arrow[from=1-1, to=2-1, "\omega_0 \times u^l_1"]
	\arrow[from=1-2, to=2-2, "\partial_0 \times \partial_2"]
	\arrow[{name=0, anchor=center, inner sep=0}, from=2-1, to=2-2, "\pi_0^{\operatorname{tor}[-1]} \times \operatorname{Id}"]
    \end{tikzcd}
    \end{equation*}
    is a pullback, where 
    \[
        \pi_0^{\operatorname{tor}[-1]} \colon [\framing \rightarrow F \rightarrow T] \longmapsto F
    \]
    and the map $ \pi_1^{\operatorname{tor}[-1]}$ sends an extension \eqref{eqn : left extension} to the sub--extension $F\rightarrow F'\rightarrow G$, which is a short exact sequence in $\fbeta$.
\end{remark}

\noindent\textbf{Left action. }In order to extract a categorical action from the above data, we would like to apply the functor \eqref{eqn : coh for qc} to the correspondence encoded by the simplicial derived stacl $\waldhausen\pairs$. However, according to the defining diagram \eqref{eqn : def of stack of stable pairs}, the stack $\pairs$ is cut out by open conditions in $\segal_2\perfstack$. In particular we can't conclude that the connected component of the stack $\pairs$ are quasi--compact. In order to state Theorem \ref{thm : left action} and Theorem \ref{thm : right action} in their full generality, we use the refinement
\[
    \pro \colon \operatorname{Corr}^\times(\operatorname{Ind(dGeom^{qc}))_{rpas,lci}} \longrightarrow \operatorname{Pro(Cat^{st}_{\infty})}
\]
where the notation $\operatorname{Ind(dGeom^{qc})}$ stands for \textit{Ind}--geometric derived stacks, and $\operatorname{Pro(Cat^{st}_{\infty})}$ is the $\infty$--category of Pro stable categories. See the appendix of \cite{Porta2019TwodimensionalCH} for the precise definition of the above functor. The shift morphism 
        \[
            [1] \ \colon \ \perfstack \longrightarrow \perfstack
        \]
        induces a natural isomorphism between the stack $\cohbetassstack$ and its image $\cohbetassstack[1]$. Thus we can apply Corollary II.3.12 in \cite{dps} to deduce the following.

 \begin{theorem}[Left action]\label{thm : left action}
      The pro--$\infty$--category $\pro(\pairs)$ is a left categorical module over the $\mathbb{E}_1$--monoidal $\infty$--category $\coh^{\operatorname{b}}(\cohbetassstack)$. In particular
       \[
            G_0(\pairs) \hspace{1cm} \text{and} \hspace{1cm} H_*^{BM}(\pairs)
       \]
       have the structure of a left module for $G_0(\cohbetassstack)$ and $H_*^{BM}(\cohbetassstack)$, respectively.
    \end{theorem}

\section{Right Action}
Throughout this section $\framing$ is a coherent sheaf on $X$.

\subsection{Construction of the right action}\label{subsec : construction of the right action}
\noindent\textbf{$\beta$--stable $\framing[1]$--copairs}
In order to construct the right action, we need to tweak the definition of stable pairs.
\begin{definition}\label{def : copairs}
     Let $\cop = [F \rightarrow T\rightarrow \framing[1]]$ be a fiber sequence in $\perf$. We say that $\cop$ in a $\beta$--stable $\framing$--copair if 
    \begin{enumerate}
        \item $F\in \fbeta$, and
        \item $T\in \tbeta$.
    \end{enumerate}
\end{definition}

\noindent The following is tautological.

\begin{lemma}\label{lemma : pair copair shift}
    Let $\framing$ be a coherent sheaf on $X$. Then the fiber sequence
    \[
        [\framing \rightarrow F \rightarrow T]
    \]
    defines a $\beta$--stable pair if and only if the rotated sequence
    \[
        [F \rightarrow T \rightarrow \framing[1]]
    \]
    defines a $\beta$--stable $\framing[1]$--copair.
\end{lemma}

\noindent\textbf{Right Hecke pattern} We introduce the stack of extensions that encodes the right action.

\begin{definition}[Extension of stable copairs]\label{def : co-extension}
    Let $G\in \cohbetass$ and let $\reflectbox{\ensuremath{\operatorname{P}}} = [F\rightarrow T \rightarrow \framing[1]]$ be a stable pair. An extension of $\cop$ by $G$ is a diagram in $\perf$ of the form
\begin{equation}\label{eqn : co-extension}
    \begin{tikzcd}
    0  \arrow[r] & G  \arrow[d] \arrow[r] & F' \arrow[d] \arrow[r] & T' \arrow[d] \\ & 0 \arrow[r] & F \arrow[d] \arrow[r] & T \arrow[d]\\ & & 0 \arrow[r] & \framing\left[1\right] \arrow[d] \\ && & 0
\end{tikzcd}
\end{equation}
all of whose squares are pullbacks. Moreover we require that $\cop'=[F'\rightarrow T'\rightarrow \framing[1]]$ is a stable copair. For an extension $\coext$ of the form \eqref{eqn : co-extension}, we denote
\[
\omega_0(\coext) = \cop \hspace{1cm} \omega_1(\coext) = \cop' \hspace{1cm} u^r_1(\coext) = G\ .  
\]
\end{definition}

\begin{proposition}[Right Hecke pattern]\label{prop : righthecke}
     Consider a diagram of the form \eqref{eqn : co-extension}. Let $G\in \cohbetass$ and let $\cop' = [F' \rightarrow T'\rightarrow \framing[1]]$ be a $\beta$--stable co--pair. Then $\cop = [F \rightarrow T \rightarrow \framing[1]]$ is a $\beta$--stable co--pair in and only if $T \in \coh(X)$.
\end{proposition}
\begin{proof} If $\cop$ is a $\beta$--stable co--pair, then $T$ is a coherent sheaf as prescribed by Definition \ref{def : copairs}. We now prove the other implication. If $T\in \coh(X)$, the fiber sequence 
    \[
        G \rightarrow T' \rightarrow T
    \]
is actually a short exact sequence in $\coh(X)$. Then, the assumption $T'\in \tbeta$ implies that $T\in \tbeta$ thanks to Lemma \ref{lemma : torsionproperties}. The fiber sequence
\[
    \framing \rightarrow F \rightarrow T
\]
induced by the diagram has its extremes in the heart $\coh(X)$. Thus we conclude that $F\in \coh(X)$. This implies that the fiber sequence
\[
    G \rightarrow F'\rightarrow  F
\]
is a short exact sequence in $\coh(X)$. Since $G\in \cohbetass\subset \fbeta$ and $F'\in \fbeta$, we conclude that that $F\in \fbeta$. In fact, let $E\subset F$ be a subsheaf. Then we have an induced short exact sequence 
\[
    0\rightarrow G \rightarrow E'\rightarrow E \rightarrow 0 \ ,
\]
where $\mu(G) = \beta$ and $\mu(E') \leq \beta$. Then we deduce that $\mu(E) \leq \beta$ .
\end{proof}
\noindent In analogy with the discussion in Subsection \ref{subsec : Construction of the left action}, we remark that the above Proposition is the essential ingredient in the proof of the associativity of the right action in Theorem \ref{theo 2}.

\bigskip

\noindent\textbf{Introduction to the section. }We now explain our plan to complete the proof of Theorem \ref{theo 2}. With a similar procedure as in the previous Sections, we define a stack $\copairs$ of stable copairs as in Definition \ref{def : copairs}, and a stack $\segal^r_1\copairs$ parametrizing extensions of stable copairs in the sense of Definition \ref{def : co-extension}. We get in this way a convolution diagram

\begin{equation}\label{eqn : convolution right}
\begin{tikzcd}
	& {\segal^r_1\copairs} &&& {\coext} & \\
	{\copairs \times \cohbetassstack} && \copairs & {(\cop,G)} && {\cop'}
	\arrow["{\omega_0\times u^r_1}", from=1-2, to=2-1]
	\arrow["{\omega_1}"', from=1-2, to=2-3]
	\arrow[from=1-5, to=2-4]
	\arrow[from=1-5, to=2-6]
\end{tikzcd}
\end{equation}
We prove further that the map $\omega_0 \times u_1^r$ is lci, and the map $\omega_1$ is proper. Then, the pull--push operation
 \[
    \omega_{1*} (\omega_0\times u^r_1)^* \colon H^{BM}_*(\copairs) \otimes H^{BM}_*(\cohbetassstack) \longrightarrow H^{BM}_*(\copairs)
 \]
 induces the structure of a right module for the \textit{CoHA} attached to the category $\cohbetass$ on the space $H^{BM}_*(\copairs)$. In particular, the associativity property of this action is encoded in the statement of Proposition \ref{prop : righthecke}. Finally, thanks to Lemma \ref{lemma : pair copair shift} we get a natural isomorphism of derived stacks
 \[
    \copairs \cong \pairs
 \]
 which allows us to get a right action on the space $H^{BM}_*(\pairs)$. We make two remarks explaining the role of the tilting procedure in our construction.
 \begin{enumerate}
     \item From a representation theoretical point of view, we get a left and a right action 
     \[
       \operatorname{HA}_{\cohbetass} \curvearrowright \bigoplus_{r,d}H^{BM}_*(\boldsymbol{\operatorname{P}}^\beta_{r,d}(X,\framing)) \curvearrowleft \operatorname{HA}_{\cohbetass}  
     \]
     where the direct sum in the central term comes from the decomposition in Remark \ref{rmk : topological type of pairs}. We unravel the actions by keeping track of the topological type. The convolution diagram \eqref{eqn : convolution left} induces a fiber sequence $F\rightarrow F'\rightarrow G$. Rotating, we get the short exact sequence in $\coh(X)$
     \begin{equation}\label{eqn : naive 1}
           0 \rightarrow G[-1] \rightarrow F \rightarrow F' \rightarrow 0
     \end{equation}
     where the topological type of $F'$ is smaller than that of $F$. Viceversa, the diagram \eqref{eqn : convolution right} induces a short exact sequence in $\coh(X)$
     \begin{equation}\label{eqn : naive 2}
         0 \rightarrow G \rightarrow F' \rightarrow F \rightarrow 0
     \end{equation}
      where the topological type of $F'$ is bigger than that of $F$. Comparing \eqref{eqn : naive 1} and \eqref{eqn : naive 2} we deduce that operators coming from the \textit{left action decrease the topological type, and operators coming from the right action increase the topological type}. Compare with Example \ref{ex : hilbS}.

      \item\label{itm : coh ss topo prop} From a geometric point of view, working in different t--structures allows us to construct morphisms with the desired properties. For example, the morphism
      \[
       \omega_0 \colon  \segal_1^l\pairs \longrightarrow \pairs, \quad \operatorname{E} \longmapsto \operatorname{P}
      \]
      is not proper, so it cannot be used to induce an operator of right action. This problem is solved by working in a different heart.
 \end{enumerate}

\subsection{Stack of right flags}\noindent We define a derived stack of copairs $\copairs$ parametrizing stable copairs. Following \cite{dps} we apply the left version of Construction I.3.6 to the 2--Segal stack $\waldhausen\perfstack$, and we get a relative 2--Segal space
\begin{equation}
    u^r_{\bullet}\ \colon \ \waldhausen^r\coflags \longrightarrow \waldhausen\perfstack
\end{equation}
As before, closed points in tthe stack $\segal_0\coflags \cong \coflags$ correspond to fiber sequences
\[
    [F \rightarrow T \rightarrow \framing[1]]
\]
in $\perf$ with no further conditions. Moreover, the stack $\segal_1\coflags$ parametrizes diagrams of the form \eqref{eqn : co-extension} without any stability condition. Moreover, we have natural maps
\[
    \begin{tikzcd}
	& {\segal_1^r\coflags} & \\
	{\coflags \times \perfstack} && \coflags
	\arrow["{\omega_0 \times u^r_1}", from=1-2, to=2-1]
	\arrow["{\omega_1}"', from=1-2, to=2-3]
\end{tikzcd}
\]
which agree with the notations in the previous Subsection. 

\bigskip

\noindent We define the derived stack of copairs via the pullback
\[\begin{tikzcd}
	\copairs & \coflags \\
	{\fbetastack \times \tbetastack} & {\perfstack \times \perfstack}
	\arrow[from=1-1, to=1-2]
	\arrow[from=1-1, to=2-1]
	\arrow["{\partial_1\times\partial_0}", from=1-2, to=2-2]
	\arrow[from=2-1, to=2-2]
\end{tikzcd}\]
As an immediate consequence of Lemma \ref{lemma : pair copair shift} we have the following.

\begin{lemma}\label{lemma : pair copair equivalence}
    The self equivalence
    \begin{equation}
        \perfstack \longrightarrow \perfstack
    \end{equation}
    which sends a fiber sequence $E_1 \rightarrow E_2 \rightarrow E_3$ to the fiber sequence $ E_2 \rightarrow E_3 \rightarrow E_1[1]$ induces an equivalence of derived stacks
    \begin{equation}
        \copairs \cong \pairs\ .
    \end{equation}
\end{lemma}

\noindent We apply again Construction I.5.5 in \cite{dps} to get a relative simplicial stack
\begin{equation}\label{eqn : relative simplicial copairs}
    u_{\bullet}^r \  \colon \ \waldhausen^r\copairs \longrightarrow \waldhausen\cohbetassstack
\end{equation}
As before, the stack $\segal_1^r\copairs$ parametrizes stable extensions in the sense of Definition \ref{def : co-extension}. In particular, the maps \eqref{eqn : convolution right} are well defined. In light to Proposition \ref{prop : righthecke}, we can use an argument analogous to that of Proposition \ref{prop : left segal} to prove the following.

\begin{proposition}\label{prop : right 2 segal}
    The relative simplicial derived stack \eqref{eqn : relative simplicial copairs} is a relative 2--Segal space.
\end{proposition}

\subsection{Proof of the right action}

\textbf{Proper push--forward. }Proceeding as in Notation III.4.22 in \cite{dps}, we can define a derived stack
\[
    \segal_1^r\boldsymbol{\operatorname{FlagCoh}}^{(1)}_{\operatorname{\beta--ss,t..f.}}(X)
\]
whose closed points are exact sequences $0 \rightarrow G\rightarrow F' \rightarrow F \rightarrow 0$ in $\coh(X)$ where $F, F'\in \fbeta$ and $G\in \cohbetass$. Let us denote as
\[
    \partial_0(G\rightarrow F' \rightarrow F) = F \quad \partial_1(G\rightarrow F' \rightarrow F) = F'\quad \partial_2(G\rightarrow F' \rightarrow F) = G \ .
\]
the forgetful maps.
\begin{lemma}\label{lemma : locally rpas in right action}
    The morphism
    \[
        \partial_1 \colon \segal_1^r\boldsymbol{\operatorname{FlagCoh}}^{(1)}_{\operatorname{\beta--ss,t..f.}}(X) \longrightarrow \fbetastack
    \]
    is locally rpas.
\end{lemma}
\begin{proof}
     Let $F'\in \fbeta$. Then the geometric fiber $\partial_1^{-1}(F')$ is a stack parametrizing short exact sequences 
        \begin{equation}\label{eqn : extension in proof}
       0 \rightarrow G\rightarrow F'\rightarrow F\rightarrow 0 \ , 
        \end{equation}
        where $F\in \fbeta$ and $G\in \cohbetass$. On the other hand, any subsheaf  $G \subset F'$ of maximal slope $\beta$ is semistable. Moreover, arguing as in the proof of Proposition \ref{prop : righthecke}, the quotient $F$ in \eqref{eqn : extension in proof} belongs to $\fbeta$. Thus, we have an identification
        \[
            \partial_1^{-1}(F') \cong \{G \in \operatorname{Quot}(X,F') \ \Big| \  \mu(G) = \beta\}
        \]
        where $\operatorname{Quot}(X,F')$ denotes the Quot scheme in the standard heart $\coh(X)$. Since the connected components of Quot--scheme $\operatorname{Quot}(X,F')$ are proper, the claim follows.
\end{proof}

\noindent Consider the map 
\begin{equation}\label{eqn : copair to tf}
    \nu_0^{\operatorname{t.f.}} :=\partial_2 \ \colon \ \copairs \longrightarrow \fbetastack
\end{equation}
which sends a stable copair $[F\rightarrow T \rightarrow \framing[1]]$ to the torsion free component $F$. Moreover, we have a morphism
\begin{equation}
    \nu_1^{\operatorname{t.f.}} \ \colon \ \segal_1^r\copairs \longrightarrow \segal_1^r\boldsymbol{\operatorname{FlagCoh}}^{(1)}_{\operatorname{\beta--ss,t..f.}}(X)
\end{equation}
sending a stable extension \eqref{eqn : co-extension} to the sub--extension $G\rightarrow F'\rightarrow F$.

 \begin{proposition}
            The square
        \begin{equation}\label{eqn : right properness square}
        \begin{tikzcd}
	{\segal_1^r \copairs} & \segal_1^r\boldsymbol{\operatorname{FlagCoh}}^{(1)}_{\operatorname{\beta--ss,t..f.}}(X)\\
	\copairs & \fbetastack
	\arrow[from=1-1, to=1-2 , "\nu_1^{\operatorname{t.f.}}"]
	\arrow[from=1-1, to=2-1, "\omega_1"]
	\arrow[from=1-2, to=2-2, "\partial_1"]
	\arrow[""{name=0, anchor=center, inner sep=0}, from=2-1, to=2-2, "\nu_0^{\operatorname{t.f.}}"]
    \end{tikzcd}
    \end{equation}
    is a pullback. In particular, the left vertical map is locally rpas.
    \end{proposition}

    \begin{proof}
        Thanks to Proposition \ref{prop : righthecke}, the proof of the first statement is analogous to that of III.4.32 in \cite{dps}. The second statement follows from the first one and Lemma \ref{lemma : locally rpas in right action}.
    \end{proof}

    \bigskip

    \noindent \textbf{Flat pull--back. } As above, we consider a stack
    \[
\segal_1^r\boldsymbol{\operatorname{FlagCoh}}^{(1)}_{\operatorname{\beta-ss,tor}}(X)
    \]
    whose closed points are short exact sequences in $\coh(X)$ of the form $0 \rightarrow G\rightarrow T' \rightarrow T \rightarrow 0$, where $T , T'\in \tbeta$ and $G\in \cohbetass$. Let us denote as
\[
    \partial_0(G\rightarrow T' \rightarrow T) = T \quad \partial_1(G\rightarrow T' \rightarrow T) = T'\quad \partial_2(G\rightarrow T' \rightarrow T) = G \ .
\]
the forgetful maps.
\begin{lemma}\label{lemma : smoothness right action}
The square
 \begin{equation}\label{eqn : aux square smoothness copairs}
         \begin{tikzcd}
	{\segal_1^r\boldsymbol{\operatorname{FlagCoh}}^{(1)}_{\operatorname{\beta-ss,tor}}(X)} & {\segal_1\cohflat(X)}\\
	\tbetastack \times \cohbetassstack & \cohflat(X) \times \cohflat(X)
	\arrow[from=1-1, to=1-2]
	\arrow[from=1-1, to=2-1]
	\arrow[from=1-2, to=2-2, "\partial_0 \times \partial_2"]
	\arrow[""{name=0, anchor=center, inner sep=0}, from=2-1, to=2-2]
    \end{tikzcd}
    \end{equation}
is a pullback. In particular, the left vertical map is quasi--compact, finitely connected and derived lci.
\end{lemma}
    \begin{proof}
    Since the horizontal maps are representable by Zariski open embeddings, we check the first statement on closed points. Let $0 \rightarrow G\rightarrow T \rightarrow T'\rightarrow 0$ be a short exact sequence of coherent sheaves on $X$, where $G\in \cohbetass$ and $T'\in \tbeta$. We claim that $T \in \tbeta$. Indeed, any quotient $T \twoheadrightarrow R$ in $\coh(X)$ induces a short exact sequence
    \[
        0 \rightarrow H \rightarrow R \rightarrow R' \rightarrow 0
    \]
    where $\mu(H) \geq \beta$ and $\mu(R') > \beta$, so that $\mu(R) > \beta$. The second statement follows from the first one, and from the fact that the right vertical morphism is derived lci. Finally, the second statement follows from the first one.
    \end{proof}
    
    \noindent Consider the map
    \begin{equation}\label{eqn : copair to tor}
       \nu^{\operatorname{tor}}_0 := \partial_1 \ \colon \ \copairs \longrightarrow \tbetastack 
    \end{equation}
    sending a stable copair $[F\rightarrow T \rightarrow \framing[1]]$ to the torsion component $T$. We also have a natural map
    \begin{equation}
        \nu^{\operatorname{tor}}_{1} \  \colon \ \segal_1^{r}\copairs \longrightarrow \segal_1^r\boldsymbol{\operatorname{FlagCoh}}^{(1)}_{\operatorname{\beta-ss,tor}}(X)
    \end{equation}
    which associates to every extension of copairs $\eqref{eqn : co-extension}$ the sub--extension $G\rightarrow T' \rightarrow T$.

    \begin{proposition}
        The square 
    \begin{equation}\label{eqn : right flatness square}
        \begin{tikzcd}
	{\segal_1^r \copairs} & {\segal_1^r\boldsymbol{\operatorname{FlagCoh}}^{(1)}_{\operatorname{\beta-ss,tor}}(X)} \\
	\copairs \times \cohbetassstack & \tbetastack \times \cohbetassstack
	\arrow[from=1-1, to=1-2 , "\nu_1^{\operatorname{tor}}"]
	\arrow[from=1-1, to=2-1, "\omega_0 \times u^l_1"]
	\arrow[from=1-2, to=2-2, "\partial_0 \times \partial_2"]
	\arrow[""{name=0, anchor=center, inner sep=0}, from=2-1, to=2-2, "\nu_0^{\operatorname{tor}}"]
    \end{tikzcd}
    \end{equation}
    is a pullback. In particular, the left vertical map is quasi--compact, finitely connected and derived lci.
    \end{proposition}

    \begin{proof}
    In light of Proposition \ref{prop : right 2 segal}, we proceed as in the proof of Lemma III.4.33 in \cite{dps} and we check the first statement on closed points. Unravelling the definitions, this amounts to proving that if we are given a fiber sequence $G\rightarrow F' \rightarrow F$ where $G\in \cohbetass$ and $F\in \fbeta$, then $F'\in \fbeta$. However, such fiber sequence is a short exact sequence in $\coh(X)$, so Lemma \ref{lemma : torsionproperties} implies the claim. The second claim follows from the first one and Lemma \ref{lemma : smoothness right action}.
    \end{proof}

\noindent\textbf{Right action. }In view of what we have already shown, we can apply Corollary II.3.14 in \cite{dps} to get the following.

    \begin{theorem}[Right action]\label{thm : right action}
        Let $\framing$ be a coherent sheaf on $X$. The pro--$\infty$--category $\pro(\copairs)$ carries the structure of a right categorical module over $\coh^{\operatorname{b}}(\cohbetassstack)$. In particular 
      \[
            G_0(\copairs) \hspace{1cm} \text{and} \hspace{1cm} H_*^{BM}(\copairs)
       \]
       have the structure of a right module for $G_0(\cohbetassstack)$ and $H_*^{BM}(\cohbetassstack)$, respectively.
    \end{theorem}

   \noindent Under the natural isomorphism of Lemma \ref{lemma : pair copair shift}, we finally get both left and right operators on the same space. 

   \begin{corollary}[Left and right action]
       Let $\framing$ be a coherent sheaf on $X$. The pro--$\infty$--category $\pro(\pairs)$ carries the structure of a left and right categorical module over $\coh^{\operatorname{b}}(\cohbetassstack)$. In particular 
      \[
            G_0(\pairs) \hspace{1cm} \text{and} \hspace{1cm} H_*^{BM}(\pairs)
       \]
       have the structure of a left and right module for $G_0(\cohbetassstack)$ and $H_*^{BM}(\cohbetassstack)$, respectively.
   \end{corollary}

   \noindent Finally, if we specialize to the case where $\framing = \linebundle[1]$ is a line bundle in $\fbeta$. Since the quot space $\Quotbetaone$ has proper connected components, we can consider the $\infty$--category $\coh^{\operatorname{b}}(\Quotbetaone)$. Thanks to the isomorphism \eqref{eqn : QuotPairs} we get a two--sided categorical action on $\coh^{\operatorname{b}}(\Quotbetaone)$. 

   \begin{corollary}
       Let $\linebundle$ be a line bundle of slope smaller that $\beta$. Then the pro--$\infty$--category \newline
       \noindent$\coh^{\operatorname{b}}(\Quotbetaone)$ carries the structure of a left and right categorical module over $\coh^{\operatorname{b}}(\cohbetassstack)$. In particular 
      \[
            G_0(\Quotbetaone) \hspace{1cm} \text{and} \hspace{1cm} H_*^{BM}(\Quotbetaone)
       \]
       have the structure of a left and right module for $G_0(\cohbetassstack)$ and $H_*^{BM}(\cohbetassstack)$, respectively.
   \end{corollary}

\newpage

\bibliographystyle{abbrv}
\bibliography{biblio}

@misc{khan1,
    author = {Adeel A. Khan},
    title = {Weaves},
    journal = {Available on the author's webpage},
    year = {2023},
    howpublished = "\url{https://www.preschema.com/papers/weaves.pdf}",
}

@article{khan2019virtualfundamentalclassesderived,
      title={Virtual fundamental classes of derived stacks I}, 
      author={Adeel A. Khan},
      year={2019},
      eprint={1909.01332},
      archivePrefix={arXiv},
      primaryClass={math.AG},
      doi={https://arxiv.org/abs/1909.01332}, 
}

@book{Huybrechts_Lehn_2010, place={Cambridge}, edition={2}, series={Cambridge Mathematical Library}, title={The Geometry of Moduli Spaces of Sheaves}, publisher={Cambridge University Press}, author={Huybrechts, Daniel and Lehn, Manfred}, year={2010}, collection={Cambridge Mathematical Library}}

@article {MR3322196,
    AUTHOR = {Negut, Andrei},
     TITLE = {Moduli of flags of sheaves and their {$K$}-theory},
   JOURNAL = {Algebr. Geom.},
  FJOURNAL = {Algebraic Geometry},
    VOLUME = {2},
      YEAR = {2015},
    NUMBER = {1},
     PAGES = {19--43},
      ISSN = {2313-1691,2214-2584},
   MRCLASS = {14J60 (14D21 14F05 19E08)},
  MRNUMBER = {3322196},
MRREVIEWER = {P.\ E.\ Newstead},
       DOI = {10.14231/AG-2015-002},
       URL = {https://doi.org/10.14231/AG-2015-002},
}

@article{Zhao_2020,
   title={On the K--Theoretic Hall Algebra of a Surface},
   volume={2021},
   ISSN={1687-0247},
   url={http://dx.doi.org/10.1093/imrn/rnaa123},
   DOI={10.1093/imrn/rnaa123},
   number={6},
   journal={International Mathematics Research Notices},
   publisher={Oxford University Press (OUP)},
   author={Zhao, Yu},
   year={2020},
    pages={4445–4486} }

@article {porta2016higheranalyticstacksgaga,
    AUTHOR = {Porta, Mauro and Yu, Tony Yue},
     TITLE = {Higher analytic stacks and {GAGA} theorems},
   JOURNAL = {Adv. Math.},
  FJOURNAL = {Advances in Mathematics},
    VOLUME = {302},
      YEAR = {2016},
     PAGES = {351--409},
      ISSN = {0001-8708,1090-2082},
   MRCLASS = {14A20 (14F05 14G22 32C35)},
  MRNUMBER = {3545934},
MRREVIEWER = {Hsian-Hua\ Tseng},
       DOI = {10.1016/j.aim.2016.07.017},
       URL = {https://doi.org/10.1016/j.aim.2016.07.017},
}

@misc{lurie2008highertopostheory,
      title={Higher Topos Theory}, 
      author={Jacob Lurie},
      year={2008},
      eprint={math/0608040},
      archivePrefix={arXiv},
      primaryClass={math.CT},
      url={https://arxiv.org/abs/math/0608040}, 
}

@article {Lin_2018,
    AUTHOR = {Lin, Yinbang},
     TITLE = {Moduli spaces of stable pairs},
   JOURNAL = {Pacific J. Math.},
  FJOURNAL = {Pacific Journal of Mathematics},
    VOLUME = {294},
      YEAR = {2018},
    NUMBER = {1},
     PAGES = {123--158},
      ISSN = {0030-8730,1945-5844},
   MRCLASS = {14D20 (14J60 14N35)},
  MRNUMBER = {3743369},
MRREVIEWER = {\'Ad\'am\ Gyenge},
       DOI = {10.2140/pjm.2018.294.123},
       URL = {https://doi.org/10.2140/pjm.2018.294.123},
}

@misc{kaushik2026cohomologyhyperquotschemescurves,
      title={The cohomology of Hyperquot schemes on curves via shifted Yangians in type A}, 
      author={Archi Kaushik},
      year={2026},
      eprint={2603.16691},
      archivePrefix={arXiv},
      primaryClass={math.AG},
      url={https://arxiv.org/abs/2603.16691}, 
}

@article {10.4134/JKMS.J230331,
    AUTHOR = {Lin, Yinbang and Wang, Sz-Sheng and Xia, Bingyu},
     TITLE = {Decorated sheaves and morphisms in tilted hearts},
   JOURNAL = {J. Korean Math. Soc.},
  FJOURNAL = {Journal of the Korean Mathematical Society},
    VOLUME = {61},
      YEAR = {2024},
    NUMBER = {6},
     PAGES = {1073--1093},
      ISSN = {0304-9914,2234-3008},
   MRCLASS = {14F08 (14D20 14D22 14D23)},
  MRNUMBER = {4816589},
MRREVIEWER = {Nicolae\ Manolache},
       DOI = {10.4134/JKMS.j230331},
       URL = {https://doi.org/10.4134/JKMS.j230331},
}

@incollection {bridgeland2020hallalgebrascurvecountinginvariants,
    AUTHOR = {Bridgeland, Tom},
     TITLE = {Hall algebras and {D}onaldson-{T}homas invariants},
 BOOKTITLE = {Algebraic geometry: {S}alt {L}ake {C}ity 2015},
    SERIES = {Proc. Sympos. Pure Math.},
    VOLUME = {97.1},
     PAGES = {75--100},
 PUBLISHER = {Amer. Math. Soc., Providence, RI},
      YEAR = {2018},
      ISBN = {978-1-4704-3577-6},
   MRCLASS = {14N35 (14J32)},
  MRNUMBER = {3821146},
MRREVIEWER = {Felix\ Janda},
       DOI = {10.1090/pspum/097.1/01670},
       URL = {https://doi.org/10.1090/pspum/097.1/01670},
}

@article{Pandharipande_2009,
   title={Curve counting via stable pairs in the derived category},
   volume={178},
   ISSN={1432-1297},
   url={http://dx.doi.org/10.1007/s00222-009-0203-9},
   DOI={10.1007/s00222-009-0203-9},
   number={2},
   journal={Inventiones mathematicae},
   publisher={Springer Science and Business Media LLC},
   author={Pandharipande, R. and Thomas, R. P.},
   year={2009},
   month=May, pages={407–447} }

@book{gaitsgory2017study,
  title={A Study in Derived Algebraic Geometry},
  author={Gaitsgory, D. and Rozenblyum, N.},
  number={Bd. 1},
  isbn={9781470435691},
  lccn={2016054809},
  series={A Study in Derived Algebraic Geometry},
  url={https://books.google.ch/books?id=b3srDwAAQBAJ},
  year={2017},
  publisher={American Mathematical Society}
}

@article{godicke2024infty,
  title={An $\infty$-Category of 2-Segal Spaces},
  author={G{\"o}dicke, Jonte},
  journal={arXiv preprint arXiv:2407.13357},
  year={2024}
}

@misc{marian2026cohomologyquotschemesmooth,
      title={The cohomology of the Quot scheme on a smooth curve as a Yangian representation}, 
      author={Alina Marian and Andrei Neguţ},
      year={2026},
      eprint={2307.13671},
      archivePrefix={arXiv},
      primaryClass={math.AG},
      url={https://arxiv.org/abs/2307.13671}, 
}

@article{shatz,
author = {Stephen S. Shatz},
title = {{Degeneration and specialization in algebraic families of vector bundles}},
volume = {82},
journal = {Bulletin of the American Mathematical Society},
number = {4},
publisher = {American Mathematical Society},
pages = {560 -- 562},
year = {1976},
}

@book{Dyckerhoff_2019,
   title={Higher Segal Spaces},
   ISBN={9783030271244},
   ISSN={1617-9692},
   url={http://dx.doi.org/10.1007/978-3-030-27124-4},
   DOI={10.1007/978-3-030-27124-4},
   journal={Lecture Notes in Mathematics},
   publisher={Springer International Publishing},
   author={Dyckerhoff, Tobias and Kapranov, Mikhail},
   year={2019} }

@article{Bayer_2021,
   title={Stability conditions in families},
   volume={133},
   ISSN={1618-1913},
   url={http://dx.doi.org/10.1007/s10240-021-00124-6},
   DOI={10.1007/s10240-021-00124-6},
   journal={Publications Mathématiques de l’IHÉS},
   publisher={MathDoc/Centre Mersenne},
   author={Bayer, Arend and Lahoz, Martí and Macrì, Emanuele and Nuer, Howard and Perry, Alexander and Stellari, Paolo},
   year={2021},
   month=jun, pages={157–325} }

@article{Rota_2021,
   title={Some Quot schemes in tilted hearts and moduli spaces of stable pairs},
   volume={32},
   ISSN={1793-6519},
   url={http://dx.doi.org/10.1142/S0129167X21500981},
   DOI={10.1142/s0129167x21500981},
   number={13},
   journal={International Journal of Mathematics},
   publisher={World Scientific Pub Co Pte Ltd},
   author={Rota, Franco},
   year={2021},
   month=oct }

@incollection {schiffmann2009lectureshallalgebras,
    AUTHOR = {Schiffmann, Olivier},
     TITLE = {Lectures on {H}all algebras},
 BOOKTITLE = {Geometric methods in representation theory. {II}},
    SERIES = {S\'emin. Congr.},
    VOLUME = {24-II},
     PAGES = {1--141},
 PUBLISHER = {Soc. Math. France, Paris},
      YEAR = {2012},
      ISBN = {978-2-85629-361-4},
   MRCLASS = {18E30 (13C60 17B37 17B67)},
  MRNUMBER = {3202707},
MRREVIEWER = {Xueqing\ Chen},
}

@misc{mmsv23,
      title={Coherent sheaves on surfaces, COHAs and deformed $W_{1+\infty}$-algebras}, 
      author={Anton Mellit and Alexandre Minets and Olivier Schiffmann and Eric Vasserot},
      year={2023},
      eprint={2311.13415},
      archivePrefix={arXiv},
      primaryClass={math.AG},
      url={https://arxiv.org/abs/2311.13415}, 
}

@misc{dhm,
      title={BPS algebras and generalised Kac-Moody algebras from 2-Calabi-Yau categories}, 
      author={Ben Davison and Lucien Hennecart and Sebastian Schlegel Mejia},
      year={2025},
      eprint={2303.12592},
      archivePrefix={arXiv},
      primaryClass={math.RT},
      url={https://arxiv.org/abs/2303.12592}, 
}

@incollection {nakajimaheisenbergalgebrahilbertschemes,
    AUTHOR = {Nakajima, Hiraku},
     TITLE = {Hilbert schemes of points on surfaces and {H}eisenberg
              algebras [translation of {S}\=ugaku {\bf 50} (1998), no.\ 4,
              385--398; MR1 690 690]},
      NOTE = {Sugaku expositions},
   JOURNAL = {Sugaku Expositions},
  FJOURNAL = {Sugaku Expositions},
    VOLUME = {15},
      YEAR = {2002},
    NUMBER = {2},
     PAGES = {207--222},
      ISSN = {0898-9583,2473-585X},
   MRCLASS = {14C05 (17B65)},
  MRNUMBER = {1944136},
MRREVIEWER = {I.\ Dolgachev},
}

@article {grojnowski1995instantonsaffinealgebrasi,
    AUTHOR = {Grojnowski, I.},
     TITLE = {Instantons and affine algebras. {I}. {T}he {H}ilbert scheme
              and vertex operators},
   JOURNAL = {Math. Res. Lett.},
  FJOURNAL = {Mathematical Research Letters},
    VOLUME = {3},
      YEAR = {1996},
    NUMBER = {2},
     PAGES = {275--291},
      ISSN = {1073-2780},
   MRCLASS = {14J60 (14C05 14J26 17B69)},
  MRNUMBER = {1386846},
MRREVIEWER = {I.\ Dolgachev},
       DOI = {10.4310/MRL.1996.v3.n2.a12},
       URL = {https://doi.org/10.4310/MRL.1996.v3.n2.a12},
}

@article{neguţ2022heckecorrespondencessmoothmoduli,
    AUTHOR = {Negu\c t, Andrei},
     TITLE = {Hecke correspondences for smooth moduli spaces of sheaves},
   JOURNAL = {Publ. Math. Inst. Hautes \'Etudes Sci.},
  FJOURNAL = {Publications Math\'ematiques. Institut de Hautes \'Etudes
              Scientifiques},
    VOLUME = {135},
      YEAR = {2022},
     PAGES = {337--418},
      ISSN = {0073-8301,1618-1913},
   MRCLASS = {14C05 (14F08 17B67)},
  MRNUMBER = {4426742},
MRREVIEWER = {Shintarou\ Yanagida},
       DOI = {10.1007/s10240-022-00131-1},
       URL = {https://doi.org/10.1007/s10240-022-00131-1},
}

@misc{diaconescuhallalgebrasonedimensional,
      title={Cohomological Hall algebras of one-dimensional sheaves on surfaces and Yangians}, 
      author={Duiliu-Emanuel Diaconescu and Mauro Porta and Francesco Sala and Olivier Schiffmann and Eric Vasserot},
      year={2025},
      eprint={2502.19445},
      archivePrefix={arXiv},
      primaryClass={math.AG},
      url={https://arxiv.org/abs/2502.19445}, 
}

@article {kapranov2022cohomologicalhallalgebrasurface,
    AUTHOR = {Kapranov, Mikhail and Vasserot, Eric},
     TITLE = {The cohomological {H}all algebra of a surface and
              factorization cohomology},
   JOURNAL = {J. Eur. Math. Soc. (JEMS)},
  FJOURNAL = {Journal of the European Mathematical Society (JEMS)},
    VOLUME = {25},
      YEAR = {2023},
    NUMBER = {11},
     PAGES = {4221--4289},
      ISSN = {1435-9855,1435-9863},
   MRCLASS = {14F08 (18G35 18M75 19L99)},
  MRNUMBER = {4662292},
MRREVIEWER = {Alexandre\ Minets},
       DOI = {10.4171/jems/1264},
       URL = {https://doi.org/10.4171/jems/1264},
}

@article {Minets_2020,
    AUTHOR = {Minets, Alexandre},
     TITLE = {Cohomological {H}all algebras for {H}iggs torsion sheaves,
              moduli of triples and sheaves on surfaces},
   JOURNAL = {Selecta Math. (N.S.)},
  FJOURNAL = {Selecta Mathematica. New Series},
    VOLUME = {26},
      YEAR = {2020},
    NUMBER = {2},
     PAGES = {Paper No. 30, 67},
      ISSN = {1022-1824,1420-9020},
   MRCLASS = {14F08 (14D23 14H60 17B37)},
  MRNUMBER = {4090584},
MRREVIEWER = {P.\ E.\ Newstead},
       DOI = {10.1007/s00029-020-00553-x},
       URL = {https://doi.org/10.1007/s00029-020-00553-x},
}

@article{ss20,
   title={Cohomological Hall algebra of Higgs sheaves on a curve},
   ISSN={2214-2584},
   url={http://dx.doi.org/10.14231/AG-2020-010},
   DOI={10.14231/ag-2020-010},
   journal={Algebraic Geometry},
   publisher={Foundation Compositio Mathematica},
   author={Sala, Francesco and Schiffmann, Olivier Schiffmann},
   year={2020},
   month=may, pages={346–376} }

@article {bradlow1993birationalequivalencesvortexmoduli,
    AUTHOR = {Bradlow, Steven B. and Daskalopoulos, Georgios D. and
              Wentworth, Richard A.},
     TITLE = {Birational equivalences of vortex moduli},
   JOURNAL = {Topology},
  FJOURNAL = {Topology. An International Journal of Mathematics},
    VOLUME = {35},
      YEAR = {1996},
    NUMBER = {3},
     PAGES = {731--748},
      ISSN = {0040-9383},
   MRCLASS = {32G13 (32L07 32L10 53C07 58D27)},
  MRNUMBER = {1396775},
MRREVIEWER = {Nicholas\ Buchdahl},
       DOI = {10.1016/0040-9383(95)00041-0},
       URL = {https://doi.org/10.1016/0040-9383(95)00041-0},
}

@article{bradlow1,
title = "Special metrics and stability for holomorphic bundles with global sections",
author = "Bradlow, Steven",
year = "1991",
month = jan,
doi = "10.4310/jdg/1214446034",
language = "English (US)",
volume = "33",
pages = "169--213",
journal = "Journal of Differential Geometry",
issn = "0022-040X",
publisher = "International Press, Inc.",
number = "1",
}

@article {bridgeland2006stabilityconditionstriangulatedcategories,
    AUTHOR = {Bridgeland, Tom},
     TITLE = {Stability conditions on triangulated categories},
   JOURNAL = {Ann. of Math. (2)},
  FJOURNAL = {Annals of Mathematics. Second Series},
    VOLUME = {166},
      YEAR = {2007},
    NUMBER = {2},
     PAGES = {317--345},
      ISSN = {0003-486X,1939-8980},
   MRCLASS = {14F05 (18E30)},
  MRNUMBER = {2373143},
MRREVIEWER = {Leovigildo\ M.\ Alonso Tarrio},
       DOI = {10.4007/annals.2007.166.317},
       URL = {https://doi.org/10.4007/annals.2007.166.317},
}

@misc{jardim2025stabilityconditionscoherentsystems,
      title={Stability conditions for coherent systems on Integral Curves}, 
      author={Marcos Jardim and Leonardo Roa-Leguizamón and Renato Vidal Martins},
      year={2025},
      eprint={2511.12610},
      archivePrefix={arXiv},
      primaryClass={math.AG},
      url={https://arxiv.org/abs/2511.12610}, 
}

@misc{dps,
      title={Cohomological Hall algebras, their categorification, and their representations via torsion pairs}, 
      author={Duiliu-Emanuel Diaconescu and Mauro Porta and Francesco Sala},
      year={2025},
      eprint={2207.08926},
      archivePrefix={arXiv},
      primaryClass={math.AG},
      url={https://arxiv.org/abs/2207.08926}, 
}

@article {macri2019lecturesbridgelandstability,
    AUTHOR = {Macrì, Emanuele and Schmidt, Benjamin},
     TITLE = {Stability and applications},
   JOURNAL = {Pure Appl. Math. Q.},
  FJOURNAL = {Pure and Applied Mathematics Quarterly},
    VOLUME = {17},
      YEAR = {2021},
    NUMBER = {2},
     PAGES = {671--702},
      ISSN = {1558-8599,1558-8602},
   MRCLASS = {14F08 (14D20 14J60)},
  MRNUMBER = {4257598},
       DOI = {10.4310/PAMQ.2021.v17.n2.a5},
       URL = {https://doi.org/10.4310/PAMQ.2021.v17.n2.a5},
}

@misc{botta2025okounkovsconjecturebpslie,
      title={Okounkov's conjecture via BPS Lie algebras}, 
      author={Tommaso Maria Botta and Ben Davison},
      year={2025},
      eprint={2312.14008},
      archivePrefix={arXiv},
      primaryClass={math.RT},
      url={https://arxiv.org/abs/2312.14008}, 
}

@misc{schiffmann2024cohomologicalhallalgebrasquivers,
      title={Cohomological Hall algebras of quivers and Yangians}, 
      author={Olivier Schiffmann and Eric Vasserot},
      year={2024},
      eprint={2312.15803},
      archivePrefix={arXiv},
      primaryClass={math.RT},
      url={https://arxiv.org/abs/2312.15803}, 
}

@misc{hausel2025pwmathcalh2,
      title={$P=W$ via $\mathcal{H}_2$}, 
      author={Tamas Hausel and Anton Mellit and Alexandre Minets and Olivier Schiffmann},
      year={2025},
      eprint={2209.05429},
      archivePrefix={arXiv},
      primaryClass={math.AG},
      url={https://arxiv.org/abs/2209.05429}, 
}

@misc{feyzbakhsh2025derivedcategorycoherentsystems,
      title={Derived category of coherent systems on curves and stability conditions}, 
      author={Soheyla Feyzbakhsh and Aliaksandra Novik},
      year={2025},
      eprint={2511.01601},
      archivePrefix={arXiv},
      primaryClass={math.AG},
      url={https://arxiv.org/abs/2511.01601}, 
}

@article {Porta2019TwodimensionalCH,
    AUTHOR = {Porta, Mauro and Sala, Francesco},
     TITLE = {Two-dimensional categorified {H}all algebras},
   JOURNAL = {J. Eur. Math. Soc. (JEMS)},
  FJOURNAL = {Journal of the European Mathematical Society (JEMS)},
    VOLUME = {25},
      YEAR = {2023},
    NUMBER = {3},
     PAGES = {1113--1205},
      ISSN = {1435-9855,1435-9863},
   MRCLASS = {14F08 (14H60 14J60 17B37 55P99)},
  MRNUMBER = {4577961},
MRREVIEWER = {Joan\ Pons-Llopis},
       DOI = {10.4171/jems/1303},
       URL = {https://doi.org/10.4171/jems/1303},
}

\end{document}